\documentclass[11pt]{article}
\usepackage{amssymb,amsthm}
\usepackage{newtxtext,newtxmath}
\usepackage{bbm}
\usepackage{setspace}
\usepackage[margin=1in]{geometry}
\usepackage{float}
 \usepackage{caption} 
\usepackage{subcaption}
\usepackage{amsmath,algorithm}
\usepackage[inline]{enumitem}
\usepackage{mathtools}
 
\DeclarePairedDelimiter\ceil{\lceil}{\rceil}

\usepackage{bbm}
\usepackage{booktabs,graphicx,multirow}
\usepackage{hyperref}
\hypersetup{colorlinks,%
	citecolor={blue}, 
	urlcolor={blue},
	linkcolor={blue},
	breaklinks={true},
	hyperindex={true},
}
\theoremstyle{plain}
\newcommand{\R}{\mathcal{R}}

\newtheorem{theorem}{Theorem}

\newtheorem{proposition}{Proposition}
\newtheorem*{assumption*}{Assumption}
\newtheorem{assumption}{Assumption}
\newtheorem{remark}{Remark}

\newtheorem*{theorem*}{Theorem}
\newtheorem{definition}{Definition}
\usepackage[toc,page]{appendix}
\allowdisplaybreaks

\usepackage{tikz}

\usepackage[
  separate-uncertainty = true,
  multi-part-units = repeat
]{siunitx}
\usepackage{tikz}
\usepackage{hyperref}
\usepackage{pgfplots}
\pgfplotsset{compat=1.17}
\pgfplotsset{every axis/.append style={
tick label style={font=\footnotesize} }}
 \definecolor{grey1}{rgb}{0.65,0.65,0.65}
\begin{document}
\title{Feedback Nash Equilibria in Differential Games with Impulse Control\footnote{We thank  Herbert Dawid for providing valuable comments and suggestions that helped us  improve the paper.}
}
\author{Utsav Sadana\thanks{ U. Sadana   is with GERAD and
McGill University, Montr\'eal, Canada (e-mail: utsav.sadana@mail.mcgill.ca).}, Puduru Viswanadha Reddy\thanks{P.V. Reddy is with the Indian Institute of Technology Madras, Chennai, India (e-mail:vishwa@ee.iitm.ac.in ).}, and Georges Zaccour 
\thanks{G. Zaccour is with 
GERAD and HEC Montr\'eal, Montr\'eal, Canada (e-mail: georges.zaccour@gerad.ca).}}

\maketitle

 \begin{abstract}
	We study a class of  deterministic  finite-horizon two-player nonzero-sum differential games where players are endowed with  different kinds of controls.   We assume that Player 1 uses piecewise-continuous controls, while Player 2 uses impulse controls. For this class of games, we seek to derive conditions for the existence of  feedback Nash equilibrium strategies for the players. More specifically, we provide a verification theorem for identifying such equilibrium strategies, using the Hamilton-Jacobi-Bellman (HJB) equations for Player 1 and the quasi-variational inequalities (QVIs) for Player 2. Further, we show that the equilibrium number of interventions  by Player 2 is upper bounded. Furthermore, we specialize the obtained results to a scalar two-player linear-quadratic differential game. In this game, Player 1's objective is to drive the state variable towards a specific target value, and Player 2 has a similar objective with a different target value. We provide, for the first time, an analytical characterization of the feedback Nash equilibrium in a linear-quadratic differential game with impulse control. We  illustrate our results using numerical experiments.

\end{abstract}

	\section{Introduction}

Many real-world applications, such as industry regulation and cybersecurity, can be modeled as a 
two-player finite-horizon nonzero-sum differential game, where one player
influences the evolution of the state variable continuously over time, whereas the other  takes actions that introduce jumps in the state variable at certain strategically chosen discrete time instants.  An example of such a
setting is a game between an environmental regulation agency, which
determines when and by how much to change the cap on pollution emissions, and a (representative)
firm,  which continuously makes production decisions that have emissions as a
by-product. %

 Nash equilibrium in differential games varies with the information that is available to the players when they determine their strategies, which is also known as the information structure \cite{Basar1999}.  In our previous paper \cite{Sadana2019}, we introduced a two-player nonzero-sum differential game with impulse controls to study the aforementioned interactions assuming an open-loop information structure, where the strategies of the players are functions of time and the initial state (which is a known parameter). It is well known that  open-loop Nash equilibrium (OLNE) strategies are not strongly time consistent, that is, that the equilibrium strategies derived for a given initial state might not constitute the equilibrium of the subgame starting at an intermediate time instant during the game, if the state value at the start of the subgame deviates from the equilibrium state trajectory determined at the start of the game \cite{Basar1989}. To address this limitation of open-loop strategies, the literature on  differential game theory has  focused on a feedback information structure, where players' actions at each instant of time during the game are determined by a strategy that depends on both the current state and the current time \cite{Haurie2012, Engwerda2005, Basar2018}. The resulting feedback strategies of the players are known to be strongly time consistent \cite{Basar1989}.  
 
The objective of this paper is to study the  class of games that we have considered in  \cite{Sadana2019}, but here under a feedback information structure. In  \cite{Sadana2019}, we have studied  a  class of differential games where Player 1 uses piecewise continuous controls and Player 2 uses impulse controls. The novelty of the present paper lies in providing conditions for the existence of a feedback Nash equilibrium (FNE) in this canonical class of differential games. We have studied these canonical games of minimal configuration for analytical tractability, and our model can be extended to the more general case where both players use both types of controls. FNE is obtained under the assumption that the impulse controls lie within the class of threshold policies, that is,  Player 2 gives an impulse only when the state leaves her continuation region, which is  characterized by using the Bensoussan Lions quasivariational inequalities (QVIs) \cite{Bensoussan1982,Bensoussan1984,Aubin1982}. Even for impulse optimal control problems, it is challenging to solve QVIs for a general class of impulse controls (see, e.g., the central bank intervention problems studied in  \cite{Cadenillas1999} and \cite{Bertola2016}). Furthermore, threshold policies are quite natural for applications in industry regulation and cybersecurity.

Our contribution is threefold:  First, we provide a verification theorem for a general class of differential games with impulse controls that can be used to characterize the FNE strategies. In particular, we show that the (value) functions that satisfy the Hamilton-Jacobi-Bellman equations for Player 1, coupled with a system of QVIs for Player 2, coincide with the  respective payoffs of the players in the FNE. The novel feature of our model is that Player 1 can continuously change both the state trajectory and Player 2's continuation set, which is a collection of all time and state vectors for which it is optimal for Player 2  not to  intervene in the system. This feature differentiates our work from the existing literature on differential games with impulse control (see \cite{Aid2020} and \cite{Basei2020}), where the continuous evolution of the state is exogenously given and all players  shift the state from one level to another at discrete time instants. Since the FNE strategies  obtained by using the verification theorem are a function of the current time and state pairs, they are strongly time consistent. 

Second, we show that, under a few regularity assumptions, the equilibrium number of impulses is bounded by a value that is derived from the problem data.

Our third contribution lies in providing, for the first time, a complete analytical characterization of FNE in a scalar linear-quadratic differential game (LQDG) with impulse controls.  LQDGs have been widely studied in
engineering, economics, and management because they provide a tractable
framework to model real-world problems involving nonconstant returns to
scale, interactions between the players' control variables, as well as
interactions between the state and control variables. LQDGs assume linear
state dynamics, which can be seen as a locally reasonable approximation of
nonlinear state dynamics. A comprehensive coverage of LQDGs can be found
in, e.g., \cite{Basar1999}, \cite{Dockner2000}, \cite{Engwerda2005}, \cite{Haurie2012}, and \cite{Basar2018}. However, these references provide
existence and uniqueness results for classical differential games, where
players only use ordinary controls and where there are no fixed costs in the
game. To the best of our knowledge, the literature on differential games
does not provide any theoretical or computational means to identify the FNE in nonzero-sum LQDGs with impulse controls. 

The specialized linear-quadratic game  we study in this paper involves  Player 1 using piecewise-continuous controls to minimize the cost associated with the state deviating from her target value, while Player 2 uses impulse controls to instantaneously change the state from one level to another so as to keep the state close to her own target. This model is a multi-agent adaptation of the impulse optimal control  problem (single player) studied in \cite{Bertola2016}.  In particular, in our setting, Player 2's impulse  optimal control problem is a modified version of the impulse control problem analyzed in \cite{Bertola2016}.  Our regularity assumptions on the value function and impulse controls of Player 2 also follow from \cite{Bertola2016} where analytical solutions of the HJB equation are obtained in the continuation region by using a quadratic form on the value function; see also \cite{Runggaldier2018}.

The remainder of the paper is organized as follows. In Section \ref{FBNEsec:lit_review}, we  review the literature on impulse optimal control problems, differential games where at least one player uses  piecewise-continuous controls, and impulse games where all players use impulse controls only. We introduce our model in Section \ref{FBNEsec:model}. In Section \ref{FBNEsec:FNE}, we provide  a verification theorem for the existence of the FNE. In Section \ref{FBNEsec:LQDG:FB}, we specialize our results  to a scalar linear-quadratic game, and we solve this game in  Section \ref{FBNEsec:numerical} for different problem parameters. Finally, concluding remarks are given in Section \ref{FBNEsec:conclusion}.

\subsection{Literature review}\label{FBNEsec:lit_review}
One of the well-studied impulse control problems   is the central bank intervention problem, where the bank intervenes in the foreign exchange market and continuously controls the domestic interest rate to keep the exchange rate close to a target value  (see, e.g., \cite{Bertola2016} and \cite{Runggaldier2018}).
The characterization of optimal impulse control in a  one-decision-maker
setting has been the topic of a long series of contributions in diverse
fields, e.g, finance \cite{Korn1998};  management  \cite{Reddy2016,Chahim2013,Chahim2017,Erdlenbruch2013,Berovic2004}; and epidemiology  \cite{Taynitskiy2019}. In contrast, the literature in
differential games with impulse controls has been very limited, and has
predominantly dealt with zero-sum games (see, e.g., \cite{Chikrii2007} and \cite{Azimzadeh2019}). With the  exception of our previous papers  \cite{Sadana2019, Sadana2020a, Sadana2020b}, the equilibrium solutions in nonzero-sum differential games with impulse controls have
been obtained under the assumption that the impulse timing is known a priori   \cite{Chang2013}.  

In \cite{Sadana2019}, we provided
an algorithm for computing the open-loop Nash equilibrium in linear-quadratic dynamic games with impulse control. Reference \cite{Sadana2020a} characterized the sampled-data Nash equilibrium for the class of games introduced in \cite{Sadana2019}. Further, \cite{Sadana2020b} determined the FNE for a specialized case of linear-state differential games
(LSDGs) with impulse controls, and showed, contrary to the case with
ordinary controls, that the FNE and OLNE do not coincide when linear value functions are used to determine the FNE. By definition, LSDGs do not account for nonlinearities in the state
variables or interactions between the state and control variables in the
players' objective functionals, which limits their applications in practice.
\ In this paper, we relax this restriction and consider a general class of differential games, and by the same token, push further
the literature in nonzero-sum differential games.

Our work is closely related to the  impulse games studied in \cite{Ferrari2019,Aid2020, Basei2020}, and \cite{Aid2021} with a feedback information structure where, however, all players are assumed to make discrete-time interventions in the continuous-time stochastic processes.  To illustrate, \cite{Ferrari2019} studied a specialized pollution control game between a government that determines the regulatory constraints on  emissions  and a (representative) firm that takes discrete-time  actions to expand its capacity. It is assumed that both the government and the firm use only impulse controls. In \cite{Aid2020}, the authors studied infinite-horizon nonzero-sum game problem assuming threshold-type impulse controls and showed that a system of QVIs gives sufficient conditions for a FNE if the value functions of both players satisfy certain regularity  conditions. There are no piecewise-continuous controls in their model, which limit its applicability to many problems of interest in regulation and security.  
Reference \cite{Basei2020} extended their two-player model to an $N$-player setting and analyzed the corresponding mean-field game. In \cite{Campi2020}, a game problem between an impulse player and a stopper is  solved using the QVIs.
The consideration of impulse controls
makes it difficult to analytically characterize Nash equilibria for a
general class of differential games, which explains why it is tempting to
focus on tractable games. For instance, \cite{Aid2020} determined closed-form solutions for symmetric linear-state
impulse stochastic games.

\section{Model}\label{FBNEsec:model}
We consider a deterministic finite-horizon two-player nonzero-sum differential game where the two players can affect a continuously evolving state vector to minimize their individual costs. In our canonical game, the two players are equipped with different types of controls. In particular, Player 1 continuously affects the state vector using her piecewise continuous control $u(t) \in \Omega_1 \subset \mathbb{R}^{m_{1}}$ while Player 2 uses discrete-time actions to instantaneously change the state by using an impulse control $\tilde{v}=\{(\tau_i,\xi_i)\}_{i\geq 1}$ where $\tau_i$ denotes an intervention instant and  $\xi_i \in \Omega_2 \subset \mathbb{R}^{m_{2}}$ denotes the size of the impulse at time $\tau_i$.  The sets $\Omega_1$ and $\Omega_2$ are assumed to be bounded and convex.

\noindent
The state vector is controlled by Player 1 and evolves as follows: 
\begin{align}
& \dot{x}(t)=f(x(t),u(t)),\;x(0^-)=x_{0},\;\text{for}\;t\neq \tau_i, i\geq 1. \label{FBNEcont_state:model} 
\end{align}
And at the impulse instant, $\tau_i$, Player 2 introduces jumps that are given by 
\begin{align}
    & x(\tau _{i}^{+})-x(\tau _{i}^{-})=g(x(\tau _{i}^{-}),\xi_{i}),  \label{FBNEjump_state:model}
\end{align}
where  $f:\mathbb{R}^{n}\times \mathbb{R}^{m_{1}}\rightarrow \mathbb{R}^{n}$, $g:\mathbb{R}^{n}\times \mathbb{R}^{m_{2}}\rightarrow \mathbb{R}^{n}$, and $\tau_i^-$ and $\tau_i^+$ denote the time instants just before and after the intervention time $\tau_i$.

\noindent
The cost functions of Player $1$ and Player $2$ are given by  
\begin{align}
&J_{1}(0,x_{0},u(.),\tilde{v})=\int_{0}^{T}h_1(x(t),u(t))dt +\sum\limits_{i\geq 1}\mathbbm{1}_{0\leq\tau_i<T}~b_1(x(\tau _{i}^{-}),\xi_{i})+s_{1}(x(T)), \label{FBNEOBJ1:model}\\
&J_{2}(0,x_{0},u(.),\tilde{v})=\int_{0}^{T}h_2(x(t),u(t))dt +\sum\limits_{i\geq 1} \mathbbm{1}_{0\leq\tau_i<T}~b_2(x(\tau _{i}^{-}),\xi_{i})+s_{2}(x(T)), \label{FBNEOBJ2:model}
\end{align}
where $h_{i}:\mathbb{R}^{n}\times \mathbb{R}^{m_{1}}\rightarrow \mathbb{R}$ is the running cost of Player $i$, $b_{i}:\mathbb{R}^{n}\times \mathbb{R}^{m_{2}}\rightarrow \mathbb{R}$ is the cost accrued by Player $i$ at the time of impulse,  and $%
s_{i}:\mathbb{R}^{n}\rightarrow \mathbb{R}$ is the terminal cost of Player $i$. Here, $\mathbbm{1}_{y}$ denotes an indicator function of $y$, that is, $\mathbbm{1}_{y}$ is equal to $1$ if $y$ holds; otherwise, it is equal to $0$.

\section{Feedback Nash equilibrium}\label{FBNEsec:FNE}
We focus our attention on the derivation of Nash equilibrium strategies
under a memoryless perfect state information structure, also referred to as feedback Nash or Markov-perfect  equilibrium. For this information structure, players use strategies that are functions of the current time $t$ and current state vector $x(t)$.
\subsection{Strategy of Player 1 and Player 2}
The strategy spaces of the players  are described as follows: Let $\Sigma:=\big\{(t,x)\, \rvert\, t\in [0,T], x \in \mathbb{R}^n\big\}$ and let $\mathcal{T}$ denote the set of admissible impulse instants.  Player 1 affects the continuously evolving state dynamics $x(t)$ using her piecewise-continuous state-feedback strategy  $\gamma:[0,T] \times \mathbb{R}^n\rightarrow \Omega_1$, while Player 2 exercises discrete-time actions  given by her state-feedback intervention policy $\delta$. Following the literature (see \cite{Bertola2016} and \cite{Aid2020}) on impulse controls, the intervention policy  $\delta$ involves determining a continuation set $\mathcal{C}$ and a continuous function $\zeta$ such that Player 2 gives an impulse if and only if $(t,x) \in \Sigma\backslash \mathcal{C}$, and when Player 2 gives an impulse, its magnitude  is given by the function $\zeta:[0,T]\times \mathbb{R}^n\rightarrow \Omega_2$. The intervention set $\mathcal{I}$ is given by $\mathcal{I}=\Sigma\backslash\mathcal{C}$.   For a given strategy pair $(\gamma,\delta)$, Player $1$'s control is given by $u(t)=\gamma(t,x)$ and Player $2$'s impulse control $\tilde{v}$ is a sequence $\{(\tau_i,\xi_i)\}_{i\geq 1}$ where $\tau_i$ is the impulse instant and $\xi_i$ is the impulse level.

 \begin{remark}
We emphasize that the timing of the interventions   are given in feedback form as the continuation set $\mathcal{C}$ depends on both the current time and the current state vector. In particular, the continuation and intervention sets will be characterized,
in Section \ref{sec:cont_interv_set}, by the QVIs associated with Player 2's optimal behavior. 
 \end{remark}

 \begin{remark}
   Nash equilibria in zero-sum differential games with impulse controls have been obtained in the literature (see, e.g., \cite{Azimzadeh2019} and \cite{Cosso2013}) assuming nonanticipative strategies \cite{Elliott1972} where each player determines her strategy as a function of her opponent's strategy in a way that the strategies do not depend on the future strategies of the opponent.  For tractability, we focus on feedback strategies that are also considered in \cite{Aid2020}. As mentioned in \cite{Aid2020}, the feedback strategies  are  dependent on the other player's strategies via the state vector, which can be affected by both the players.
\end{remark}

 \begin{remark}\label{def:admissible}
  The actions of the players associated with an admissible strategy pair $(\gamma,\delta)$ can be described as follows: Player 1 continuously controls the state trajectory using state feedback $\gamma(t,x)$ during the time that the state lies in the continuation set $\mathcal{C}$. When the state leaves  set $\mathcal{C}$, Player 2 intervenes and gives an impulse of size $\zeta(t,x)$ to bring the state into set $\mathcal{C}$. 
 \end{remark}
 
\begin{definition}\label{FBNEassum_interior:LQDGFB} The sequence $\tilde{v}=\{(\tau
_{i},v_{i})\}_{i\geq 1},$ is an admissible impulse control  of Player $2$ if  the number of impulses is finite and the impulse instants lie in the set $\mathcal {T}$ given by 
	\begin{align*}
&\mathcal T=\{\tau_i,\,i\geq 1,\,\,~\rvert~0\leq  \tau _{1}<\tau _{2}<\cdots< T\},\\
&\tau_n=\inf\{t>\tau_{n-1}:(t,x)\not\in \mathcal{C}\},\, n\geq 1,\, \tau_0:=0.
	\end{align*}
\end{definition}

\noindent
The above definition ensures that Player 2 gives an impulse as soon as the state leaves the continuation set $\mathcal{C}$.\\

\noindent
Next, we determine the cost-to-go functions for Player 1 and Player 2 for a given strategy pair $(\gamma,\delta)$ and for any starting position of the game $(t,x)$. Suppose $\gamma_{[t,T]} \in \Gamma_{[t,T]}$ and $\delta_{[t,T]}\in \Delta_{[t,T]}$ are restrictions of $\gamma$ and $\delta$, respectively, to the interval $[t,T]$, and $\Gamma_{[t,T]}$ and $\Delta_{[t,T]}$ denote the strategy sets for Player 1 and Player 2, respectively, in the interval $[t,T]$.  Then, the state evolution for any starting position of the game $(t,x)$ is given by
\begin{align}
& \dot{x}(t)=f(x(t),\gamma(t,x(t))),\;x(t)=x,\;\text{for}\;(t,x)\in \mathcal{C},  \label{FBNEcont_state:strategies} \\
& x(\tau_i^+)-x(\tau_i^{-})=g(x(\tau_i^{-}),\xi_i),\;\text{for}%
\;(\tau_i,x(\tau_i)) \in \mathcal{I},   \label{FBNEjump_state:strategies}
\end{align}
and the  cost-to-go functions are given by       
\begin{align}
&J_{1}(t,x,\gamma_{[t,T]},\delta_{[t,T]})=\int_{t}^{T}h_1(x(s),\gamma_{[t,T]}(s,x(s)))ds+\sum\limits_{j\geq 1} \mathbbm{1}_{t\leq\tau_j<T}~ b_1(x(\tau _{j}^{-}),\xi_j)+s_{1}(x(T)), \label{FBNEOBJ1:strategies}\\
&J_{2}(t,x,\gamma_{[t,T]},\delta_{[t,T]})=\int_{t}^{T}h_2(x(s),\gamma_{[t,T]}(s,x(s)))ds +\sum\limits_{j\geq 1} \mathbbm{1}_{t\leq\tau_j<T}~b_2(x(\tau _{j}^{-}),\xi_j)+s_{2}(x(T)), \label{FBNEOBJ2:strategies}
\end{align}
The differential game described by  \eqref{FBNEcont_state:strategies}-\eqref{FBNEOBJ2:strategies} constitutes a nonstandard optimal control problem of Player 1 due to intervention costs and state jumps, and an impulse optimal control problem of Player 2.

\noindent
 The feedback Nash equilibrium  is defined as follows:
	\begin{definition}\label{def:FNE}
		For the differential game described by (\ref{FBNEcont_state:strategies}--\ref{FBNEOBJ2:strategies}) with a memoryless perfect state information pattern, the strategy profile $(\gamma^*,\delta^*)\in \Gamma\times \Delta$ constitutes a feedback Nash equilibrium solution if, for any $(t,x)\in \Sigma$, we have 
		\begin{subequations}
		\begin{align}
	&	J_1(t,x,\gamma_{[t,T]}^*,\delta_{[t,T]}^*)\leq J_1(t,x,\gamma_{[t,T]},\delta_{[t,T]}^*), \;\forall \gamma_{[t,T]} \in \Gamma_{[t,T]},\label{FNE:P1} \\
	&		J_2(t,x,\gamma_{[t,T]}^*,\delta_{[t,T]}^*)\leq J_2(t,x,\gamma_{[t,T]}^*,\delta_{[t,T]}), \;\forall \delta_{[t,T]} \in \Delta_{[t,T]}. \label{FNE:P2}	\end{align}
	\end{subequations}
	\end{definition}
\subsection{Verification theorem}\label{FBNEsubsec:verification:theorem}
In this section, we provide methods for identifying the FNE associated with the differential game described
by (\ref{FBNEcont_state:strategies}--\ref{FBNEOBJ2:strategies}). To this end, from \eqref{FNE:P1}, we know that the FNE strategy $\gamma^*$ of Player 1 provides the best response
to Player 2's FNE strategy $\delta^*$. Similarly, from \eqref{FNE:P2}, Player 2's FNE strategy $\delta^*$ is the best response to Player
1's FNE strategy $\gamma^*$. Further,  $V_1:[t,T]\times \mathbb{R}^n \rightarrow \mathbb{R}$ and $V_2:[t,T]\times \mathbb{R}^n \rightarrow \mathbb{R}$ denote the equilibrium cost-to-go of the players in the subgame
starting at $(t,x) \in \Sigma$, and can be defined  as follows using \eqref{FNE:P1} and \eqref{FNE:P2}:
\begin{subequations}
\begin{align}
    &V_1(t,x) =\inf_{\gamma_{[t,T]} \in \Gamma_{[t,T]}} J_1(t,x,\gamma_{[t,T]},\delta^*_{[t,T]}),\\
      &V_2(t,x) =\inf_{\delta_{[t,T]} \in \Delta_{[t,T]}} J_2(t,x,\gamma^*_{[t,T]},\delta_{[t,T]}).
\end{align}
\end{subequations}
The following is a standing assumption on the value functions, which will be used
throughout the paper.
\begin{assumption}\label{assum:Vi}
The value function of Player 1, $V_1(t,x)$, is   differentiable in both $t$ and $x$ when $(t,x) \in \mathcal{C}$.
\end{assumption}

From \eqref{FNE:P1}, the value function $V_1(t,x)$ associated with Player 1's optimal control
problem satisfies the following Hamilton-Jacobi-Bellman (HJB) equation for a given impulse control $\{(\tau_i^*,\xi_i^*)\}_{i\geq 1}$ corresponding to  Player 2's FNE strategy $\delta^*$:
\begin{subequations}\label{HJB_all}	
	\begin{align}
   & -\frac{\partial V_1(t,x) }{\partial t}= \min_{\varphi \in \Omega_{1}} \mathcal{H}_1\left(t,\varphi,\frac{\partial V_1(t,x)}{\partial x}\right), \,(t,x) \in \mathcal{C},\label{FBNEeq:HJB:V1}\\
   & V_1(T,x(T))=s_1(x(T)),\, \forall (T,x) \in \Sigma, \label{FBNEeq:terminal:V1}\\
   &    V_1(\tau_i^{*-},x(\tau_i^{*-}))=V_1(\tau_i^{*-},x(\tau_i^{*-})+g(x(\tau_i^{*-}),\xi_i^*))+ b_1(x(\tau_i^{*-}),\xi_i^*),\, (\tau_i^*,x(\tau_i^*))\in  \mathcal{I},\label{FBNEeq:update:V1}
\end{align}
where
\begin{align}
\mathcal{H}_1\left(t,\varphi,\frac{\partial V_1(t,x)}{\partial x}\right)=    h_{1}(x,\varphi)+ \left(\frac{\partial V_1(t,x)}{\partial x}\right)^Tf(x,\varphi).
\end{align}
\end{subequations}

The above conditions  can be interpreted as follows. From Definition \ref{FBNEassum:lipschitz}, an admissible impulse cannot occur at the terminal time, hence condition \eqref{FBNEeq:terminal:V1} holds. In the continuation region $\mathcal{C}$, Player 2 does not give any impulse, and therefore, the value function of Player 1 satisfies the HJB equation \eqref{FBNEeq:HJB:V1}. When an impulse occurs in the intervention region, that is, $(\tau_i^*,x(\tau_i^*)) \in \mathcal{I}$, then Player 1's cost-to-go is the sum of the additional cost, $b_1(x(\tau_i^{*-}),\xi_i^*)$,  incurred due to the intervention by Player 2, and the cost-to-go from playing optimally afterwards.

\begin{remark}
   We remark that the discontinuities in Player 1's value function  can occur only due to interventions by Player 2.
\end{remark}
\subsection{Continuation and intervention set}\label{sec:cont_interv_set}
Player 2 solves the impulse optimal control problem \eqref{FNE:P2} for a given equilibrium strategy $\gamma^*(t,x)$ of Player 1.
\begin{assumption}
The value function of Player 2, $V_2(t,x)$, is differentiable in both $t$ and $x$ for almost all values of the state $x$.\footnote{This assumption is also made  in \cite{Bertola2016} to define (weak) QVIs (see \eqref{FBNEQVI:HJB}).}
\end{assumption}

The value function $V_2(t,x)$  associated with Player 2's impulse control problem 
satisfies the following system of (weak) QVIs
\begin{subequations}\label{eq:QVI_all}
\begin{align}
 &\frac{\partial V_2(t,x)}{\partial t}+\mathcal{H}_2\left(x,\gamma^*(t,x), \frac{\partial V_2(t,x)}{\partial x}\right)\geq 0,\, \forall t\in [0,T],\; \text{a.a. } x \in \mathbb{R}^n,\label{FBNEQVI:HJB}\\
&\forall (t,x) \in \Sigma, \text{ the following two relations hold:}\notag\\
& \qquad V_2(t,x)\leq \mathcal{R}V_2(t,x), \label{FBNEQVI:v:greater:rv}\\
& \qquad  (V_2(t,x)-\mathcal{R}V_2(t,x))\Big(\frac{\partial V_2(t,x)}{\partial t}+\mathcal{H}_2\left(x,\gamma^*(t,x), \frac{\partial V_2(t,x)}{\partial x}\right)\Big)=0, \label{FBNEQVI:complementarity}\\
& \text{and } V_2(T,x)=s_2(x(T)), \forall (T,x)\in \Sigma, \label{FBNEQVI:terminal}
\end{align}
where the Hamiltonian operator $\mathcal{H}_2$ and intervention operator $\mathcal{R}$ are defined as follows:
\begin{align}
  &  \mathcal{H}_2\left(x,\gamma^*(t,x), \frac{\partial V_2(t,x)}{\partial x}\right)=h_2(x,\gamma^*(t,x))+\left(\frac{\partial V_2(t,x)}{dx}\right)^T f(x,\gamma^*(t,x)),\\
   &  \mathcal{R}V_2(t,x)=\min_{\eta \in \Omega_2}V_2(t,x+g(x,\eta))+b_2(x,\eta).\label{eq:operatorR}
\end{align}
\end{subequations}
\begin{remark}
QVIs can be interpreted as follows:

\begin{enumerate}
    \item Condition \eqref{FBNEQVI:v:greater:rv} ensures that the value function $V_2(\cdot)$ evaluated at any $(t,x)\in \Sigma$ is at most equal to the
minimum cost that Player 2 incurs from intervening  at time $t$ and playing optimally afterwards.

\item  Player $2$ does not intervene at a time $t$ if the cost-to-go from giving an impulse at time $t$ is strictly greater than the value function $V_2(\cdot)$ evaluated at $(t,x)\in \Sigma$. Hence, when $V_2(t,x)=\mathcal{R}V_2(t,x)$, Player 2 gives
an impulse.

\item  At any $(t,x)\in \Sigma$, condition \eqref{FBNEQVI:complementarity} ensures that either Player 2 waits so that the HJB-like equation
\eqref{FBNEQVI:HJB} for Player 2 holds with equality or Player 2 gives an impulse.
\end{enumerate}
\end{remark}

\begin{remark}
The value function of Player 2, $V_2(t,x)$, can have kinks at those  time instants when the state value is at the boundary of the continuation set $\mathcal{C}$. In (single-agent) impulse control problems, the value function is assumed to be  differentiable throughout the time horizon (see \cite{Bertola2016}, \cite{Aid2020}, and the references therein).
\end{remark}

\begin{remark}
The condition $V_2(\tau,x)=\mathcal{R}V_2(\tau,x)$ results in the continuity of the value function of Player 2 at the impulse instant $\tau$ under the feedback information structure. For impulse control problems studied by using the Pontryagin maximum principle,  the Hamiltonian continuity condition \cite{Chahim2013} gives the timing of interventions   (see also \cite{Sadana2019}, where differential games with impulse control are analyzed using the impulse version of the Pontryagin maximum principle).
\end{remark}
QVIs allow us to define the continuation
and intervention sets for Player 2 as follows:
\begin{definition}
The continuation and intervention sets are given by
\begin{align}
    &\mathcal{C}=\Big\{(t,x) \in \Sigma \rvert V_2(t,x)<\mathcal{R} V_2(t,x), \frac{\partial V_2(t,x)}{\partial t}+\mathcal{H}_2\left(x,\gamma^*(t,x), \frac{\partial V_2(t,x)}{\partial x}\right) = 0\Big\},\\
     &   \mathcal{I}=\Big\{(t,x)\in \Sigma \rvert V_2(t,x)=\mathcal{R} V_2(t,x), \frac{\partial V_2(t,x)}{\partial t}+\mathcal{H}_2\left(x,\gamma^*(t,x), \frac{\partial V_2(t,x)}{\partial x}\right) \geq 0\Big\}.
\end{align}
\end{definition}

\begin{remark}
   In impulse games studied in \cite{Aid2020} and \cite{Basei2020},  the system of QVIs for any  player $j$ has an additional intervention operator to account for  impulses by the other player(s), while the Hamiltonian operator is not an explicit function of the strategies of other player(s). In our game problem, the Hamiltonian operator of Player 2 depends on the strategies of Player 1, which in turn continuously affects the continuation and intervention sets of Player 2. Further, in the infinite-horizon impulse game studied in \cite{Aid2020}, the continuation sets depend only on the current state.
\end{remark}
\begin{assumption}\label{FBNEassum:Unique_impulse}
There exists a unique measurable function $\zeta:[0,T]\times \mathbb{R}^n\rightarrow \Omega_2$  such that
\begin{align}\label{FBNEeq:optimal_impulse}
    \zeta(t,x)=\arg \min_{\eta\in \Omega_2} \{V_2(t,x+g(x,\eta))+b_2(x,\eta)\}.
\end{align}
\end{assumption}
Here, \eqref{FBNEeq:optimal_impulse} gives the optimal impulse level at any $(t,x)$ since it minimizes the sum of the immediate cost ($b_2(x,\eta)$) incurred from giving an impulse of size $\eta$ and the  cost-to-go from playing optimally afterwards (see also \cite{Aid2020}, where a similar assumption is used to solve stochastic impulse games).

\noindent
We have the following assumptions regarding the state dynamics \eqref{FBNEcont_state:model}--\eqref{FBNEjump_state:model} and the objective functions described by \eqref{FBNEOBJ1:model}--\eqref{FBNEOBJ2:model}:

\begin{assumption}\label{FBNEassum:lipschitz}
The state dynamics and objective functions of Player 1 and Player 2 satisfy the following conditions:
\begin{enumerate}
\item \label{FBNEassum:lipschitz:f} $f(x,u)$ is  (uniformly) Lipschitz continuous in $x$, that is, there exists a constant $c_f>0$, such that
\begin{align*}
    &\lvert f(x,u)-f(y,u)\rvert \leq c_f\lvert x-y \rvert,\; \forall x,y \in \mathbb{R}^n,\;u\in \Omega_1.
\end{align*}
\item \label{FBNEassum:lipschitz:g} $g(x,\xi)$ is (uniformly) Lipschitz continuous in $x$, such that, for $c_g>0$, we have
\begin{align*}
    &\lvert g(x,\xi)-g(y,\xi)\rvert \leq c_g \lvert x-y \rvert\; \forall x,y \in \mathbb{R}^n,\;\xi\in \Omega_2.
\end{align*}
\item \label{FBNEassum:lipschitz:cost} $\forall x\in \mathbb{R}^n$,  $\inf_{\eta\in \Omega_2}b_2(x,\eta)=\mu>0.$
\item \label{FBNEassum:bounds}The functions $f$, $g$, $h_i$, $b_i$ and $s_i$ are bounded for $i\in \{1,2\}$.
\end{enumerate}
\label{FBNEassum:dynamics}
\end{assumption}

Assumptions 1.\ref{FBNEassum:lipschitz:f} and 1.\ref{FBNEassum:lipschitz:g} ensure that there exists a unique state trajectory $x(\cdot)$ for any measurable $u(\cdot)$ and impulse sequence $\{(\tau_i,\xi_i)\}_{i\geq 1}$.  Assumption 1.\ref{FBNEassum:lipschitz:cost} ensures that Player $2$ intervenes only a finite number of times in the game due to the fixed cost associated with each impulse (see \cite{Bertola2016}, where similar assumptions are provided in the context of an impulse optimal control problem). Assumption 1.\ref{FBNEassum:bounds} is used later to show that  the value functions of Player 1 and Player 2 have an upper and lower bound that depend  on the problem parameters.

\noindent

The sufficient conditions to characterize the FNE of the differential game described in \eqref{FBNEcont_state:strategies}-\eqref{FBNEOBJ2:strategies} are given in the next theorem.

\begin{theorem}[Verification Theorem] \label{FBNEthm:verif}  Let Assumptions \ref{assum:Vi}-\ref{FBNEassum:lipschitz} hold. Suppose there exist functions $V_i:[0,T]\times \mathbb{R}^n\rightarrow \mathbb{R} (i=1,2)$ such that $V_1(t,x)$ satisfies the
HJB equations \eqref{HJB_all} and $V_2(t,x)$ satisfies   the QVIs \eqref{eq:QVI_all} for all $(t,x)\in \Sigma$. Suppose there exist
strategies $(\gamma^*,\delta^*)$ with the following properties. Player 1's control $u^*(t)=\gamma^*(t,x)$ satisfies for all $t\in [0,T]$
\begin{subequations}
\begin{align}
    u^*(t)=\gamma^*(t,x)=\arg\min_{\varphi\in \Omega_1}\mathcal{H}_1\left(x,\varphi,\frac{\partial V_1(t,x)}{\partial x}\right), \label{eq:thm:ustar}
\end{align}
and Player 2's impulse control is a sequence $\{(\tau_j^*,\xi_j^*)\}_{j\geq 1}$ where interventions occur at $\tau_j^*=t$ if $(t,x)\in \mathcal{I}$, that is, $(t,x)$ satisfy
\begin{align}
    V_2(t,x)=\mathcal{R}V_2(t,x),
\end{align}
and the corresponding impulse levels $\xi_j^*$ are given by
\begin{align}
    \xi_j^*=\zeta(t,x)=\arg\min_{\eta\in \Omega_2}\left(V_2(t,x+g(x,\eta))+b_2(x,\eta)\right).\label{eq:thmimpulse}
\end{align}
\end{subequations}
Then, $(\gamma^*,\delta^*)$ is a FNE of the differential game described by (\ref{FBNEcont_state:strategies}--\ref{FBNEOBJ2:strategies}).  Further, $V_i(t,x)$  is the equilibrium cost-to-go of Player $i$, $(i=1,2)$ for the subgame starting at $(t,x) \in \Sigma$ and defined over the horizon $[t,T]$.
\end{theorem}
\begin{proof}
From Definition \ref{def:FNE}, we have to show that
\begin{align*}
  &  V_j(t,x)=J_j(t,x,\gamma^*_{[t,T]},\delta^*_{[t,T]}), \, j=\{1,2\},\\
   & V_1(t,x)\leq J_1(t,x,\gamma_{[t,T]},\delta^*_{[t,T]}), \forall \gamma_{[t,T]} \in \Gamma_{[t,T]},\\
   & V_2(t,x)\leq J_2(t,x,\gamma^*_{[t,T]},\delta_{[t,T]}), \forall \delta_{[t,T]} \in \Delta_{[t,T]}.
\end{align*}
Suppose $x_1(\cdot)$ is the state trajectory generated by Player 1 using an arbitrary admissible strategy $\gamma_{[t,T]}$ and Player 2 using the strategy $\delta_{[t,T]}^*$ such that Player 1's control $u(t)$ is given by $u(s)=\gamma_{[t,T]}(s,x(s))$, $s\in [t,T]$. Using the total derivative of $V_1(\cdot)$ between the impulse instants $(\tau_{j-1}^*,\tau_{j}^*)$, integrating with respect to $t$ from $\tau_{j-1}^*$ to $\tau_{j}^*$, and  taking the summation for all $j\geq 1$, we obtain
\begin{align*}
V_1(T,x_1(T))-V_1(t,x)= &\sum_{j\geq 1}\int_{ \tau_{j-1}^{*+}}^{ \tau_{j}^{*-}}\Big(\frac{\partial V_1}{\partial t}(s,x_1(s)) + \left(\frac{\partial V_1}{\partial x}(s,x_1(s))\right)^Tf(x_1(s),u(s))\Big)ds\\& \quad+\sum_{j\geq 1}\mathbbm{1}_{t\leq \tau_j^*< T}(V_1(\tau_j^{*+},x_1(\tau_j^{*+}))-V_1(\tau_j^{*-},x_1(\tau_j^{*-}))),
\end{align*}
where we defined $\tau_0^*:=t$.
 From \eqref{FBNEeq:HJB:V1}, we know that, for any given control $u(s)$, the following inequality holds:
\begin{align}
   &\frac{\partial V_1}{\partial t}(s,x_1(s)) + \left(\frac{\partial V_1}{\partial x}(s,x_1(s))\right)^Tf(x_1(s),u(s))\geq -h_1(x_1(s),u(s)).
\end{align}
Therefore, we obtain
\begin{align*}
V_1(T,x_1(T))-V_1(t,x)&\geq -\sum_{j\geq 1}\int_{ \tau_{j-1}^{*+}}^{ \tau_{j}^{*-}} h_1(x_1(s),u(s))ds+\sum_{j\geq 1}\mathbbm{1}_{t\leq \tau_j^*< T}(V_1(\tau_j^{*+},x_1(\tau_j^{*+}))-V_1(\tau_j^{*-},x_1(\tau_j^{*-}))).
\end{align*}
From the terminal condition \eqref{FBNEeq:terminal:V1} on $V_1(\cdot)$ and \eqref{FBNEeq:update:V1}, we obtain
\begin{align*}
    V_1(t,x)&\leq  s_{1}(x(T))+ \sum_{j\geq 1}\int_{ \tau_{j-1}^{*+}}^{ \tau_{j}^{*-}}h_1(x_1(s),u(s))ds+\sum_{j\geq 1}\mathbbm{1}_{t\leq \tau_j^*< T}~ b_1(x_1(\tau _{j}^{*-}),\xi_j^*)\\&\quad=J_1(t,x,{\gamma}_{[t,T]},\delta_{[t,T]}^*).
\end{align*}
For a strategy $\gamma^*$ of Player 1, \eqref{eq:thm:ustar} holds for the equilibrium control $u^*(t)$ of Player 1, so we obtain
\begin{align*}
    V_1(t,x)&=  s_{1}(x^*(T))+ \sum_{j\geq 1}\int_{ \tau_{j-1}^{*+}}^{ \tau_{j}^{*-}}h_1(x^*(s),u^*(s))ds+\sum_{j\geq 1}\mathbbm{1}_{t\leq \tau_j^*< T}~b_1(x^*(\tau _{j}^{*-}),\xi_j^*)\\&\quad=J_1(t,x,{\gamma^*}_{[t,T]},\delta^*_{[t,T]}),
\end{align*}
where $x^*$ is the state trajectory generated by Player 1 choosing the strategy $\gamma^*_{[t,T]}$ and Player 2 choosing the strategy $\delta^*_{[t,T]}$.
Therefore, $\gamma^*_{[t,T]}$ is the best response to Player 2's strategy $\delta^*_{[t,T]}$.

\noindent
Next, we consider an arbitrary admissible strategy $\delta_{[t,T]}$ of Player 2 such that the intervention instants are given by $\tau_i,\, i\geq 1$ and the corresponding impulse levels are given by $\xi_i$. Further, $x_2(\cdot)$ is the state trajectory generated by the strategy pairs $(\gamma^*_{[t,T]},\delta_{[t,T]})$. We obtain the following relation by taking the total derivative of $V_2(\cdot)$ between the impulse instants $(\tau_{j-1},\tau_{j})$, integrating over time from $\tau_{j-1}$ to $\tau_{j}$, and taking the summation for all $j\geq 1$:
\begin{align}
V_2(T,x_2(T))-V_2(t,x)&= \sum_{j\geq 1}\int_{ {\tau}^{+}_{j-1}}^{ {\tau}_{j}^{-}}\Big(\frac{\partial V_2(s,x_2(s))}{\partial s} + \left(\frac{\partial V_2(s,x_2(s))}{\partial x}\right)^Tf(x_2(s),\gamma^*(s,x_2(s)))\Big)ds\notag\\&\quad+\sum_{j\geq 1}\mathbbm{1}_{t\leq {\tau}_j< T}(V_2({\tau}_j^+,x_2({\tau}_j^+))-V_2({\tau}_j^-,x_2({\tau}_j^-))). \label{FBNEproof:veri:Taylor}
\end{align}
The value function satisfies \eqref{FBNEQVI:HJB} for all $(t,x) \in \Sigma$, so we have 
\begin{align}
   & \frac{\partial V_2}{\partial t}(s,x_2(s)) + \left(\frac{\partial V_2}{\partial x}(s,x_2(s))\right)^Tf(x_2(s),\gamma^*(s,x_2(s)))\geq -h_2(x_2(s),\gamma^*(s,x_2(s)). \label{FBNEeq1}
\end{align}
Given an impulse of size, $\xi_j$, and from the definition of an intervention operator given in \eqref{eq:operatorR}, we obtain $$\R V_2({\tau}_j^-,x_2({\tau}_j^-))\leq V_2({\tau}_j^+,x_2({\tau}_j^{+}))+b_2(x_2({\tau}_j^-),{\xi_j}).$$
Also, from \eqref{FBNEQVI:v:greater:rv}, we know that $$\R V_2({\tau}_j^-,x_2({\tau}_j^-))-V_2({\tau}_j^-,x_2({\tau}_j^-))\geq 0.$$ Therefore, we obtain
  \begin{align}
    V_2({\tau}_j^+,x_2({\tau}_j^{+}))-V_2({\tau}_j^-,x_2({\tau}_j^-))&\geq\R V_2({\tau}_j^-,x_2({\tau}_j^-))-V_2({\tau}_j^-,x_2({\tau}_j^-))-b_2(x_2({\tau}_j^-),{\xi_j})\geq  -b_2(x_2({\tau}_j^-),{\xi_j}).\label{FBNEeq2}
\end{align}
Substitute \eqref{FBNEeq1}
and \eqref{FBNEeq2} in \eqref{FBNEproof:veri:Taylor} to obtain \begin{align*}
    V_2(T,x_2(T))-V_2(t,x)&\geq  \sum_{j\geq 1}\int_{ {\tau}_{j-1}^+}^{ {\tau}_{j}^-} -h_2(x_2(s),\gamma^*(s,x(s))ds-\sum_{j\geq 1}\mathbbm{1}_{t\leq {\tau}_j< T}~b_2(x({\tau}_j^-),{\xi_j}).
\end{align*}
Substituting the terminal condition $V_2(T,x_2(T))=s_2(x_2(T))$, given in \eqref{FBNEQVI:terminal}, in the above inequality yields
\begin{align*}
 V_2(t,x)&\leq  s_2(x_2(T))+ \sum_{j\geq 1}\int_{ {\tau}_{j-1}^+}^{ {\tau}_{j}^-} h_2(x_2(s),\gamma^*(s,x_2(s))ds\\&\quad +\sum_{j\geq 1}\mathbbm{1}_{t\leq {\tau}_j< T}~b_2(x_2({\tau}_j^-),\xi_j)=J_2(t,x,\gamma_{[t,T]}^*,\delta_{[t,T]}).
\end{align*}
The strategy $\delta^*$ of Player 2 entails giving impulses at $\tau_j^*=t$ where the pair $(t,x(t))\in \Sigma$ is such that \eqref{FBNEQVI:v:greater:rv} holds with equality, and the corresponding impulse levels $\xi_j^*$ satisfy \eqref{eq:thmimpulse}. Therefore,  for a strategy $\delta^*$, we obtain
\begin{align*}
    &V_2({\tau}_j^{*+},x^*({\tau}_j^{*+}))-V_2({\tau}_j^{*-},x^*({\tau}_j^{*-}))=-b_2(x^*({\tau}_j^{*-}),{\xi_j^*}),\\
   & \frac{\partial V_2}{\partial t}(s,x^*(s)) + \left(\frac{\partial V_2}{\partial x}(s,x^*(s))\right)^Tf(x^*(s),\gamma^*(s,x^*(s)))= -h_2(x^*(s),\gamma^*(s,x^*(s)),
\end{align*}
and the cost-to-go function is given by
\begin{align*}
  V_2(t,x)&= s_2(x^*(T))+ \sum_{j\geq 1}\int_{ {\tau}_{j-1}^{*+}}^{ {\tau}^{*-}_{j}} h_2(x^*(s),\gamma^*(s,x^*(s))ds+\sum_{j\geq 1}b_2(x^*({\tau}_j^{*-}),\xi_j^*)\\&\quad=J_2(t,x,\gamma_{[t,T]}^*,{\delta^*_{[t,T]}}).
\end{align*}
Therefore, $\delta^*$ is the best response strategy to Player 1's strategy $\gamma^*$.
\end{proof}
\begin{remark}
An important feature of the FNE solution introduced in Definition \ref{def:FNE} is that if the strategy pair $(\gamma^*,\delta^*)$
provides a FNE to differential game described by (\ref{FBNEcont_state:strategies}--\ref{FBNEOBJ2:strategies}) with duration $[0,T]$, then its restriction to the time
interval $[t,T]$, denoted by $(\gamma^*_{[t,T]},\delta_{[t,T]}^*)$, provides a FNE to the same differential game defined on the shorter
time interval $[t,T]$, with any initial state $x(t)$. Since, this property holds true for all $0\leq t\leq T$ and for all  state values $x(t)$, the
FNE $(\gamma^*,\delta^*)$ is strongly time consistent.
\end{remark}

\noindent
Next, we show that there can only be a finite number of impulses during the game.
\begin{proposition}\label{FBNEprop_bound_number}
Let Assumption \ref{FBNEassum:lipschitz} hold. Then, the value functions of Player 1 and Player 2 are bounded. The equilibrium number of impulses $K\in \mathbb{N}$ is  bounded by
\begin{align}\label{eq:bound:number:impulse}
   K =\ceil {\frac{2\left(T\|h_2\|_{\infty}+\|s_2\|_{\infty}\right)}{\mu}},
\end{align}
where $\mu=\inf_{\eta \in \Omega_2}b_2(x,\eta)>0,\; \forall x\in \mathbb{R}^n$, and $\ceil{y}$ denotes the smallest integer that is greater than or equal to $y$.
\end{proposition}
\begin{proof}
See Appendix \ref{FBNEproof:prop:number:upper}. 
\end{proof}

QVIs have been solved in the literature under some restrictive assumptions on the value functions, even for games with linear objective functions, see e.g., \cite{Aid2020} and \cite{Campi2020}. An additional difficulty in our case is that the QVIs are coupled with HJB equations associated with Player 1's best response.  In the next section, we specialize our results to linear-quadratic differential games and  provide a complete analytical characterization of the FNE strategies. 

\section{A scalar linear-quadratic differential game with targets}\label{FBNEsec:LQDG:FB}

In this section, we consider a scalar linear-quadratic adaptation of the differential game (\ref{FBNEcont_state:model}-\ref{FBNEOBJ2:model}), referred to as iLQDG hereafter. Player 1 and Player 2 aim to minimize the costs resulting from the deviation of the state away from their target state values $\rho_1$ and $\rho_2$, respectively. In our model,  the structure of Player 2's problem (objective functions and state dynamics) is  inspired by the impulse optimal control problem analyzed in \cite{Bertola2016}.
\begin{subequations}\label{FBNEiLQDG}
\begin{align}
   & \text{(iLQDG): }\notag\\
    &J_{1}(0,x_{0},u(\cdot),\tilde{v}) =\int_{0}^{T}\frac{1}{2}\left(
w_{1}(x(t)-\rho_1)^2+r_{1}u(t)^2\right) dt +\sum\limits_{i=1}^{k}z_1\lvert \xi_i  \rvert +\frac{1}{2}s_{1}(x(T)-\rho_1)^2,\\
&J_{2}(0,x_{0},u(\cdot),\tilde{v}) =\int_{0}^{T}\frac{1}{2}
w_{2}(x(t)-\rho_2)^2 dt 
+\sum\limits_{i=1}^{k}h(\xi_i)  +\frac{1}{2}s_{2}(x(T)-\rho_2)^2,\\
    &\dot{x}(t)=ax(t)+bu(t),\;x(0^{-})=x_{0},\,\forall t\neq \tau_i,\, i\geq 1,  \label{FBNEeq:LQDG:FBstateq}\\
 & x(\tau_i^{+})=x(\tau_i^{-})+\xi_i, \label{FBNEeq:LQDG:FB_jump}
\end{align}
\label{FBNEeq:iLQDG:FB}
\end{subequations}
where 
\begin{align}
    h(\xi_i):=\begin{cases} C+c \xi_i& \text{if }\, \xi_i>0\\
    \min (C,D) & \text{if }\, \xi_i=0\\
    D-d\xi_i& \text{if }\,\xi_i<0,
    \end{cases}
\end{align}
and $w_1,\, r_1,\,z_1,\, s_1,\, w_2,\, s_2,\, C,\,D,\,c,\,d$ are positive constants.\\

\noindent
In the above iLQDG, the impulse can be positive, negative, or $0$. Each intervention results in fixed  costs, equal to $C$ or $D$, for Player 2, even if the magnitude of the impulse at the intervention instant $\tau_i$ is $0$. Player 1 incurs a positive cost $z_1\lvert \xi_i\rvert$ due to interventions by Player 2. We can view $z_1\lvert \xi_i\rvert$ as the cost associated with the disruption of Player 1's resources due to Player 2's actions.

The continuation set of Player 2 is described in the following assumption  (see also \cite{Bertola2016} and \cite{Runggaldier2018}):

\begin{assumption}\label{assum:continuationset:ilqdg}
Player $2$ gives an impulse if $(t,x)$ does not lie in the continuation set $\mathcal{C}$ given by
\begin{eqnarray}
    \mathcal{C}=\{(t,x)\in \Sigma~ \rvert~\ell_1(t)<x<\ell_2(t)\}.
\end{eqnarray}
Player 2 shifts the state to $\alpha(t)$ if $x\leq\ell_1(t)$, and to $\beta(t)$ if $x \geq \ell_2(t)$, so that the following relation holds:
\begin{eqnarray}
    \ell_1(t)<\alpha(t)<\beta(t)<\ell_2(t).
\end{eqnarray}
\end{assumption}
 The threshold policy of Player 2 involves determining the boundaries $\ell_1(\cdot)$ and $\ell_2(\cdot)$ of the continuation region $\mathcal{C}$ as well as the values  $\alpha(\cdot)$ and $\beta(\cdot)$, to which Player 2 shifts the state whenever the state reaches the boundaries $\ell_1(\cdot)$ or $\ell_2(\cdot)$, respectively.  The functions $\ell_1(\cdot)$,  $\alpha(\cdot)$, $\beta(\cdot)$, and $\ell_2(t)$ are obtained from the QVIs.
 
 \begin{assumption}\label{assum:feedback:Player1}
The state feedback strategy of Player 1 defined in the continuation set $\mathcal{C}$ is given by $\gamma(t,x)=p_1(t)x+q_1(t)$ where the real valued functions $p_1:[0,T]\rightarrow \mathbb R$ and $q_1:[0,T]\rightarrow \mathbb R$ are continuous.
 \end{assumption}
 
It is to be noted that the above assumption allows for discontinuities in the control of Player 1 at the impulse instants due to the corresponding jumps in the state. However, for a given state value in the continuation set $\mathcal{C}$,  Player 1's strategy is continuous in $t$ and $x$.

\noindent
We make the following assumption on the admissible controls of Player 1 and Player 2:
\begin{assumption}\label{assum:interior}
The admissible control $u$ of Player 1  and impulse size $\xi$ for Player 2 lie in the interior of the bounded and open convex sets $\Omega_1$ and $\Omega_2$, respectively
\end{assumption}
\subsection{Optimal control problem of Player 1}
Let the equilibrium strategy of Player 2 be given by $\delta^*$ such that Player 2 gives an impulse if the state leaves the continuation set $\mathcal{C}$ described in Assumption \ref{assum:continuationset:ilqdg}.  Then, the equilibrium strategy of Player 1 can be determined by finding the value function that satisfies \eqref{FBNEeq:HJB:V1}-\eqref{FBNEeq:update:V1} for the iLQDG. 

Player 1 solves a linear-quadratic optimal control problem in the continuation region $\mathcal{C}$, and at the impulse instant $\tau_i$, Player 1's cost is given by $z_1 \lvert \xi_i \rvert$. Therefore, we can make the following guess on the form of the value function of Player 1:
\begin{assumption}\label{assum:V1form}
The value function of Player 1 is given by:
\begin{align}\label{FBNEstructure:V1}
    V_1(t,x)=\begin{cases}
    \Phi_1(t,\alpha(t))+z_1\lvert\alpha(t)-x\rvert &x \leq \ell_1(t)\\
  \Phi_1(t,x) & x\in (\ell_1(t),\ell_2(t))\\
     \Phi_1(t,\beta(t))+z_1\lvert \beta(t)-x\rvert & x \geq \ell_1(t)
     \end{cases}.
    \end{align}
Since the game is  linear-quadratic,   $\Phi_1$  is quadratic in the state: 
\begin{align}
  \Phi_1(t,x)=\frac{1}{2}p_1(t)x^2+q_1(t)x+n_1(t). \label{FBNEV1def_periodic}
\end{align}    
\end{assumption}
The equilibrium control of Player 1 is obtained by substituting the value function in the HJB equation.
From \eqref{FBNEeq:HJB:V1}, we have
\begin{align}
    -\frac{\partial \Phi_1(t,x) }{\partial t}=& \min_{u \in \Omega_{u}} \Big(\frac{1}{2}w_{1}(x-\rho_1)^2+\frac{1}{2}r_{1}u(t)^2+ \left(\frac{\partial \Phi_1}{\partial x}\right)(ax+bu(t)) \Big). \label{FBNEeq:HJB1:periodic}
\end{align}
Differentiating the right-hand side of the above equation and equating the
result to zero yields the equilibrium strategy of Player 1 (see Assumption \ref{assum:interior}):
\begin{align}
   \gamma^*(t,x)= u^*(t) =-\frac{b}{r_1}\left(\frac{\partial \Phi_1}{\partial x}\right)=-\frac{b}{r_1}(p_1(t)x+q_1(t)). \label{FBNEeq:u:periodic}
\end{align}
Substituting \eqref{FBNEeq:u:periodic} in the state dynamics \eqref{FBNEeq:LQDG:FBstateq}, we obtain  
 \begin{align}
 \dot{x}(t)&=ax(t)+bu^*(t)
 =ax(t)-\frac{b^2}{r_1}\left( p_1(t)x(t)+q_1(t)\right)\notag\\
 &=\left(a-\frac{b^2}{r_1} p_1(t)\right)x(t) -\frac{b^2}{r_1}q_1(t)\notag\\&=a_x(t)x(t) +b_x q_1(t), \label{FBNEeq:stateC:periodic}
 \end{align}
 where $a_x(t)=a-\frac{b^2}{r_1}p_1(t)$ and $b_x=-\frac{b^2}{r_1}$.
On substituting \eqref{FBNEeq:u:periodic} and \eqref{FBNEV1def_periodic} in \eqref{FBNEeq:HJB1:periodic}, we obtain
\begin{align*}
-\frac{1}{2} \dot{p}_1(t)x^2- \dot{q}_1(t) x-\dot{n}_1(t)&=
    \frac{1}{2}
w_{1}(x-\rho_1)^2-\frac{1}{2}
b_x(p_1(t)x+q_1(t))^2+\left(p_1(t)x+q_1(t)\right) (a_x(t)x+b_xq_1(t))\\
   \Rightarrow -\dot{p}_1(t)x^2- 2\dot{q}_1(t) x-2\dot{n}_1(t)&=
w_{1}x^2+w_1\rho_1^2-2xw_1\rho_1-
b_x\left(p_1(t)^2x^2-q_1(t)^2\right)+ 2a_x(t)\left(p_1(t)x+q_1(t)\right)x.
\end{align*}
Upon rearranging a few terms in the above equation, we get  
\begin{align*}
     &\left(\dot{p}_1(t)+w_1+b_x p_1(t)^2+2p_1(t)a\right)x^2+w_1\rho_1^2+2\dot{n}_1(t)+ b_x q_1(t)^2 +\left(2\dot{q}_1(t)+ 2a_x(t)q_1(t)-2w_1\rho_1\right)x=0.
\end{align*}
Since the above equation must hold for all $x$ except at $(t,x) \not\in \mathcal{C}$,   $p_1(\cdot )$, $q_1(\cdot )$, and $n_1(\cdot )$ evolve as follows: 
\begin{subequations}\label{FBNEeq:exo:P1:dyn:periodic}
\begin{align}
   & \dot{p}_1(t)=-w_1-b_x p_1(t)^2-2p_1(t)a, \label{FBNEeq:p1:dyn}\\
   & \dot{q}_1(t)=-a_x(t)q_1(t)+w_1\rho_1,\,\label{FBNEeq:q1:dyn}\\
  &  \dot{n}_1(t)=-\frac{1}{2}b_x q_1(t)^2-\frac{w_1\rho_1^2}{2},\, \label{FBNEeq:n1:dyn}
\end{align}
\end{subequations}
where $p_1(T)=s_1,\, q_1(T)=-s_1\rho_1$, and $n_1(T)=\frac{1}{2}s_1\rho_1^2$.

The solution of \eqref{FBNEeq:p1:dyn} is given by the following equation (see Appendix \ref{FBNEsolve31}):
\begin{align}\label{FBNEp1:closedform}
   p_1(t)= \frac{1}{b_x}\left(-a+\frac{\theta}{2}-\frac{\theta}{C_{1}e^{\theta t}+1 }\right).
\end{align}

\noindent
Using the value of $p_1(t)$ given in \eqref{FBNEp1:closedform}, we obtain
\begin{align*}
    a_x(t)= a+b_xp_1(t)= \frac{\theta}{2}-\frac{\theta}{C_{1}e^{\theta t}+1 }.
\end{align*}

\begin{proposition}
    Let Assumptions \ref{assum:continuationset:ilqdg}-\ref{assum:V1form} hold. Then, the equilibrium state-feedback strategy of Player 1 is given by
\begin{align}
    \gamma^*(t,x)=\left(\frac{\theta}{2}-\frac{\theta}{C_{1}e^{\theta t}+1} \right)x-\frac{b^2}{r_1}q_1(t),\, \forall (t,x) \in \mathcal{C},
\end{align}
where
\begin{align}\label{FBNEC1}
&\theta= 2\sqrt{a^2+w_1\frac{b^2}{r_1}},\\
 & C_{1}= \left(\frac{2\theta}{\theta+2\frac{b^2}{r_1}s_1-2a}-1\right)e^{-\theta T}.
\end{align}
\end{proposition}
\noindent
When an impulse occurs, that is, $(\tau_i^{*},x(\tau_i^{*})) \in\mathcal{I}$, it follows from \eqref{FBNEeq:update:V1} that $V_1$ satisfies
\begin{align*}
\frac{1}{2}  p_1(\tau_i^{*-})x^2+q_1(\tau_i^{*-})x+n_1(\tau_i^{*-})&=\frac{1}{2}  p_1(\tau_i^{*+})(x+\xi_i^*)^2+q_1(\tau_i^{*+})(x+\xi_i^*)+n_1(\tau_i^{*+})+z_1\lvert \xi_i^* \rvert.
\end{align*}
The equilibrium strategy of Player 2 is to bring the state to $\alpha(t)$ if $x\leq \ell_1(t)$, and to $\beta(t)$  if $x\geq \ell_2(t)$, that is,  $x+\xi_i^*=\alpha(\tau_i^{*-})$ if  $x\leq \ell_1(t)$ and $x+\xi_i^*=\beta(\tau_i^{*-})$ if $x\geq \ell_2(t)$. Therefore, we have
\begin{align*}
    \frac{1}{2}  p_1(\tau_i^{*-})x^2+q_1(\tau_i^{*-})x+n_1(\tau_i^{*-})&=\frac{1}{2}  p_1(\tau_i^{*+})\alpha(\tau_i^{*-})^2+q_1(\tau_i^{*+})\alpha(\tau_i^{*-})+n_1(\tau_i^{*+})\\
 &\quad+z_1\lvert \alpha(\tau_i^{*-})-x \rvert,\, x\leq \ell_1(t),\\
 \frac{1}{2}  p_1(\tau_i^{*-})x^2+q_1(\tau_i^{*-})x+n_1(\tau_i^{*-})
 &=\frac{1}{2}  p_1(\tau_i^{*+})\beta(\tau_i^{*-})^2+q_1(\tau_i^{*+})\beta(\tau_i^{*-})+n_1(\tau_i^{*+})\\
 &\quad+z_1\lvert \beta(\tau_i^{*-})-x\rvert,\, x\geq \ell_2(t).
\end{align*}
Since $\xi_i^*=\alpha(\tau_i^{*-})-x>0$ and $\xi_i^*=\beta(\tau_i^{*-})-x<0$, we have
\begin{align*}
&\frac{1}{2}  p_1(\tau_i^{*-})x^2+q_1(\tau_i^{*-})x+n_1(\tau_i^{*-})=\frac{1}{2}  p_1(\tau_i^{*+})\alpha(\tau_i^{*-})^2+(z_1+q_1(\tau_i^{*+}))\alpha(\tau_i^{*-})+n_1(\tau_i^{*+})-z_1x,\, x\leq \ell_1(t),\\
& \frac{1}{2}  p_1(\tau_i^{*-})x^2+q_1(\tau_i^{*-})x+n_1(\tau_i^{*-})=\frac{1}{2}  p_1(\tau_i^{*+})\beta(\tau_i^{*-})^2+(-z_1+q_1(\tau_i^{*+}))\beta(\tau_i^{*-})+n_1(\tau_i^{*+})+z_1x,\, x\geq \ell_2(t).
\end{align*}
The above equations and continuity of $p_1$ and $q_1$ (from Assumption \ref{assum:feedback:Player1}) imply that, at the impulse instants, the following conditions are satisfied:
\begin{align*}
&n_1(\tau_i^{*-})=n_1(\tau_i^{*+}) +\frac{1}{2}p_1(\tau_i^{*}) \alpha(\tau_i^{*-})^2-\frac{1}{2}p_1(\tau_i^{*})x^2 +(z_1+q_1(\tau_i^{*}))\alpha(\tau_i^{*-})-(q_1(\tau_i^{*-})+z_1)x, x\leq \ell_1(t),\\
&n_1(\tau_i^{*-})=n_1(\tau_i^{*+}) +\frac{1}{2}p_1(\tau_i^{*}) \beta(\tau_i^{*-})^2-\frac{1}{2}p_1(\tau_i^{*})x^2 -(z_1-q_1(\tau_i^{*}))\beta(\tau_i^{*-})-(q_1(\tau_i^{*-})-z_1)x, x\geq \ell_2(t).
\end{align*}
\subsection{Impulse control problem of Player 2}
Player 2 solves the QVIs associated to her impulse control problem for a given equilibrium strategy $\gamma^*$ of Player $1$.

In the continuation region, Player 2's running cost is quadratic in the state, and it is  is linear in the state  in the intervention region. Therefore, we can make the following conjecture  on the form of the value function of Player 2:
\begin{assumption}\label{assum:formV2}
The value function of Player 2 is given by 
\begin{align}\label{FBNEstructure:V2}
    V_2(t,x)=\begin{cases}
    \Phi_2(t,\alpha(t))+C+c(\alpha(t)-x) &x \leq \ell_1(t)\\
  \Phi_2(t,x) & x\in (\ell_1(t),\ell_2(t))\\
     \Phi_2(t,\beta(t))+D+d(x-\beta(t)) & x \geq \ell_2(t),
     \end{cases}
    \end{align}
    where
    \begin{align}\label{FBNEV2_linear:quad}
    \Phi_2(t,x)=\frac{1}{2}p_2(t)x^2+q_2(t)x+n_2(t).
\end{align}
\end{assumption}  
A similar assumption on the form of the value function was made in \cite{Bertola2016} to obtain analytical solutions for an  impulse optimal control problem.

The value function $V_2$ coincides with continuous and continuously differentiable function $\Phi_2$  in the continuation region $\mathcal{C}$. We conjecture that $\Phi_2$ is quadratic in state  because  the cost functions are quadratic in state. In the intervention region, the value function is equal to the sum of the intervention cost incurred by the player to shift the state to the continuation region and the cost-to-go (which is equal to $\Phi_2(t,\alpha(t)$ or $\Phi_2(t,\beta(t)$ depending on the state value at the impulse time) from playing optimally afterwards.

\noindent    
    When the state lies in the continuation region, that is, $x\in (\ell_1(t), \ell_2(t))$, the value function of Player $2$ satisfies \eqref{FBNEQVI:HJB} with equality
\begin{align*}
    \frac{\partial \Phi_2(t,x) }{\partial t}+ \left(\frac{\partial \Phi_2}{\partial x}\right)(ax+b\gamma^*(t,x))+  \frac{1}{2}
w_{2}(x-\rho_2)^2=0.
\end{align*}

Substituting the partial derivatives of $\Phi_2(t,x)$ and the equilibrium control of Player $1$ from \eqref{FBNEeq:u:periodic} in the above equation yields
 \begin{align*}
&\frac{1}{2} \dot{p}_2(t) x^2+ \dot{q}_2(t) x+\dot{n}_2(t)+\left(p_2(t)x+q_2(t)\right)a_x(t)x-w_{2}x\rho_2\\&\qquad+b_xq_1(t)\left(p_2(t)x+q_2(t)\right)+\frac{1}{2}w_{2}x^2+\frac{1}{2}w_{2}\rho_2^2=0.
 \end{align*}
   On comparing the coefficients, we obtain
   \begin{subequations}
   \label{FBNEeq:exo:P2:dyn:periodic}
 \begin{align}
 \dot{p}_2(t)&=-w_2-2p_2(t)\left(\frac{\theta}{2}-\frac{\theta}{C_{1}e^{\theta t}+1 }\right),\; \label{FBNEeq:p2:dyn} \\
 \dot{q}_2(t)&=-\left(\frac{\theta}{2}-\frac{\theta}{C_{1}e^{\theta t}+1 }\right)q_2(t)-b_xp_2(t)q_1(t)+w
 _2\rho_2,\label{FBNEeq:q2:dyn}\\
 \dot{n}_2(t)&=-b_xq_1(t)q_2(t)-\frac{1}{2}w_2\rho_2^2,
 \end{align}
    \end{subequations}
   where $p_2(T)=s_2,\, q_2(T)=-s_2\rho_2$, and $n_2(T)=\frac{1}{2}s_2\rho_2^2$.
   
\noindent
The solution of \eqref{FBNEeq:p2:dyn} is given by
\begin{align} p_2(t)=\frac{-w_{2}{{e}}^{2t\theta }C_{1}^2-2t\theta w_{2}{{e}}^{t\theta }C_{1}+w_{2}+H\theta {{e}}^{t\theta }}{\theta {\left(C_{1}{{e}}^{t\theta }+1\right)}^2},
\end{align}
where 
\begin{align}
    H=&2C_{1}s_{2}+s_{2}{{e}}^{-T\theta }-\frac{w_{2}{{e}}^{-T\theta }-C_{1}^2w_{2}{{e}}^{T\theta }}{\theta }+C_{1}^2s_{2}{{e}}^{T\theta }+2C_{1}Tw_{2}.
\end{align}

\noindent
We make the following assumption on the problem parameters so that $p_2>0$ for all $t\in [0,T]$, and consequently, the value function of Player 2 is strictly convex in the continuation region $\mathcal{C}$.
\begin{assumption}\label{FBNEV2:convex}
For $t\in [0,T]$, the problem parameters satisfy
\begin{align}
 H\theta+ w_{2}(1-{\mathrm{e}}^{t\theta }C_{1}^2-2t\theta C_{1}) >0.  
\end{align}
\end{assumption}

\subsubsection{Intervention set and continuation set}

 In the intervention region ($(t,x)\in\mathcal{I}$), \eqref{FBNEQVI:v:greater:rv} holds with equality, that is, 
\begin{align}
    V_2(t,x)=\mathcal{R}V_2(t,x)=\min_{\eta\in \Omega_2} (V_2(t,x+\eta)+h(\eta)).
\end{align}
For the problem parameters assumed in this section, $V_2$ is  strictly convex in $x$ (see Assumption \ref{FBNEV2:convex}) and continuously differentiable for $y=x+\eta \in \mathcal{C}$. Since $\alpha(t),\beta(t) \in \mathcal{C}$, and $x+\eta$ takes a value of $\alpha(t)$ or $\beta(t)$ at the intervention instants and the derivative of $y$ with respect to $\eta$ is equal to 1, we can use the first-order conditions to obtain
\begin{align}
    \frac{\partial \Phi_2(t,\alpha(t))}{\partial y} +\frac{\partial h(\eta)}{\partial \eta}=0,\; x\leq \ell_1(t),\\
      \frac{\partial \Phi_2(t,\beta(t))}{\partial y} +\frac{\partial h(\eta)}{\partial \eta}=0,\; x \geq \ell_2(t).
\end{align}
Using the quadratic form of the value function in \eqref{FBNEV2_linear:quad} for the state value in the continuation region $(\ell_1(t), \ell_2(t))$, we get
\begin{eqnarray}\label{FBNEalpha2diffquad}
\begin{aligned}
  &&  \frac{\partial \Phi_2}{\partial y}(t,\alpha(t))=p_2(t)\alpha(t)+q_2(t)=-c,\\
  &&  \frac{\partial \Phi_2}{\partial y}(t,\beta(t))=p_2(t)\beta(t)+q_2(t)=d.
  \end{aligned}
\end{eqnarray}
\begin{subequations}
Therefore, the following functions $\alpha(\cdot)$ and $\beta(\cdot)$ give the state values after an impulse occurs at equilibrium:
\begin{align}
&\alpha(t)=-\frac{q_2(t)+c}{p_2(t)},\forall t \in [0,T],\label{FBNEeq:alpha}\\
&\beta(t)=\frac{d-q_2(t)}{p_2(t)},\, \forall t\in [0, T].\label{FBNEeq:beta}
\end{align}
\end{subequations}
Since \eqref{FBNEQVI:v:greater:rv} holds with equality in the intervention region, we have
\begin{eqnarray}
    V_2(t,x)=\begin{cases}V_2(t,\alpha(t))+C+c(\alpha(t)-x) &x \leq \ell_1(t)\\
    V_2(t,\beta(t))+D+d(x-\beta(t)) & x \geq \ell_2(t).
    \end{cases}
\end{eqnarray}
Also, $\alpha(t)$ and $\beta(t)$ lie in the continuation region $\mathcal{C}$, which implies $V_2(t,\alpha(t))=\Phi_2(t,\alpha(t))$ and $V_2(t,\beta(t))=\Phi_2(t,\beta(t))$. For $x=\ell_1(t)$ and $x=\ell_2(t)$, we substitute \eqref{FBNEV2_linear:quad} in the above equations and simplify to obtain
\begin{subequations}
\begin{align}
    &\frac{1}{2}p_2(t)\ell_1(t)^2+q_2(t)\ell_1(t)=\frac{1}{2}p_2(t)\alpha(t)^2+q_2(t)\alpha(t)+C+c(\alpha(t)-\ell_1(t)),\label{FBNEeq:V2:continuous:alpha}\\
      &\frac{1}{2}p_2(t)\ell_2(t)^2+q_2(t)\ell_2(t)=\frac{1}{2}p_2(t)\beta(t)^2+q_2(t)\beta(t) +D+d(\ell_2(t)-\beta(t)). \label{FBNEeq:V2:continuous:beta}
\end{align}
\end{subequations}
To characterize the left boundary of the continuation region, we substitute
 $\alpha(t)$ in \eqref{FBNEeq:V2:continuous:alpha}  to get
\begin{align*}
&p_2(t)\ell_1(t)^2+2(q_2(t)+c)\ell_1(t)-p_2(t)\left(-\frac{q_2(t)+c}{p_2(t)}\right)^2-2(q_2(t)+c)\left(-\frac{q_2(t)+c}{p_2(t)}\right)-2C=0 \\
    &\Rightarrow p_2(t)\ell_1(t)^2+2(q_2(t)+c)\ell_1(t)+\frac{(q_2(t)+c)^2}{p_2(t)}-2C=0. 
\end{align*}
Since $C>0$, $p_2(t)>0$, and $\ell_1(t)<\alpha(t)$, the left boundary of the continuation region is given by
\begin{subequations}
\begin{align}
\ell_1(t)=\frac{-c-q_2(t) - \sqrt{2Cp_2(t)}}{p_2(t)}.\label{FBNEeq:ell1}
\end{align}
On substituting $\beta(t)$ in \eqref{FBNEeq:V2:continuous:beta}, we obtain the right boundary of the continuation region
\begin{align*}
&p_2(t)\ell_2(t)^2+2(q_2(t)-d)\ell_2(t)-p_2(t)\left(\frac{d-q_2(t)}{p_2(t)}\right)^2-2(q_2(t)-d)\left(\frac{d-q_2(t)}{p_2(t)}\right)-2D=0\\
    &\Rightarrow p_2(t)\ell_2(t)^2+2(q_2(t)-d)\ell_2(t)+\frac{(d-q_2(t))^2}{p_2(t)}-2D=0.
\end{align*}
From  $D>0, p_2(t)>0$ and $\ell_2(t)>\beta(t)$, we obtain
\begin{align}
    \ell_2(t)=\frac{-q_2(t)+d + \sqrt{2Dp_2(t)}}{p_2(t)}.\label{FBNEeq:ell2}
\end{align}
\end{subequations}

 By construction, $V_1(t,x)$ satisfies the sufficient conditions in \eqref{HJB_all}, and therefore, $V_1$ is a value function of Player 1.   In the next theorem, we give conditions under which $V_2(t,x)$ in \eqref{FBNEstructure:V2} satisfies the QVIs \eqref{eq:QVI_all}.
\begin{theorem}\label{theorem:verif:lqdg}
Let Assumptions \ref{assum:continuationset:ilqdg}-\ref{FBNEV2:convex} hold. $V_2(t,x)$ in \eqref{FBNEstructure:V2} is the value function of Player 2 if $\ell_1(t)\leq x_{11}(t)$ and $\ell_2(t)\geq x_{22}(t)$ for each $t\in[0,T]$ where $\ell_1(t)$ and $\ell_2(t)$ are given in \eqref{FBNEeq:ell1} and \eqref{FBNEeq:ell2}, respectively, 
   \begin{subequations}
   \begin{align}
   &x_{11}(t)=\frac{(ca+w_2\rho_2)-\sqrt{\theta_\alpha(t)} }{w_2},\\
  & x_{22}(t)=\frac{-(da-w_2\rho_2)+\sqrt{\theta_\beta(t)} }{w_2},\\
   & \theta_\alpha(t)=c^2a^2+2w_2\left(ca\rho_2-\frac{\partial \Phi_2(t,\alpha(t))}{\partial t}\right),\label{FBNEeq:delalpha}\\
  & \theta_\beta(t)= d^2a^2-2w_2\left(da\rho_2-\frac{\partial \Phi_2(t,\beta(t))}{\partial t}\right),\label{FBNEeq:delbeta}
   \end{align}
   \end{subequations}
   and $x_{11}(t)$ and $x_{22}(t)$ are well defined with $\theta_{\alpha}(t)\geq0$ and $\theta_{\beta}(t)\geq 0$ for all $t\in [0,T]$.
\end{theorem}
\begin{proof}
See Appendix \ref{proof:theorem:verif:lqdg}.
\end{proof}
\section{Numerical examples} \label{FBNEsec:numerical}
To illustrate our results, we consider an iLQDG with time horizon $T=1$ and other problem parameters given in Table \ref{FBNEtab:Parameters}. 

\begin{table}[h]
    \centering
    \begin{tabular}{ccccccccccccccc}
    \toprule
      $a$&$b$&$w_1$&$s_1$&$r_1$&$z_1$&$w_2$& $s_2$&$c$&$C$&$D$&$d$&$\rho_1$ &$\rho_2$\\
     \midrule
$0.1$&$-0.3$&$1$&$1$&$1$&$2$&$4$&$1$&$2$&$3$&$5$&$3$&$2.5$&$5$\\
     \bottomrule
    \end{tabular}
    \caption{Parameters for numerical example }
    \label{FBNEtab:Parameters}
\end{table}

In Figure \ref{fig:s2}, we provide a complete characterization of the state feedback policy of Player 2 for the problem parameters in  Table \ref{FBNEtab:Parameters}. Player 2 gives an impulse at any time $t$ if the state reaches a level $\ell_1(t)$ or lower and brings the state to $\alpha(t)$. If the state reaches a level $\ell_2(t)$ or higher, then Player 2 gives an impulse to bring the state to $\beta(t)$. Since the cost coefficient $s_2$ of the salvage value for Player 2 is lower than the running cost coefficient $w_2$, the functions  $\ell_2(t)$, $\beta_2(t)$, $\ell_1(t)$, and $\alpha(t)$ diverge over time away from the target state value $\rho_2=5$. Also, the fixed cost and the marginal cost of intervention are small if the state crosses the lower boundary compared to the case when the state crosses the upper boundary. As a result, $\lvert \ell_1(t)- \alpha(t)\rvert < \lvert \ell_2(t)- \beta(t)\rvert$ for all $t\in [0,T]$. For  initial state values of $2$, $5$, and $8$, the evolution of equilibrium state trajectories is given by $x_1^*(t)$, $x_2^*(t)$, and $x_3^*(t)$, respectively; see Figure \ref{fig:s2}. The equilibrium strategies are strongly time consistent which implies that if the state deviates from the equilibrium path such that the state value $x(t)$ is  below $\ell_1(t)$ or  above $\ell_2(t)$ at any $t\in (0,T)$, Player 2 brings the state to $\alpha(t)$ or $\beta(t)$, respectively; this observation is illustrated in  Figure \ref{fig:s2}. 
 
 \begin{figure}[htbp]
    \centering
 \begin{tikzpicture}
    \begin{axis}[ xlabel=$t$,xmin=0,xmax=1,ymin=0,legend style={overlay},  legend columns=4,
		legend pos=north west,scale=1]
		\addplot[dashed,gray!30, ultra thick] file {Data/alpha_nofix_s2.dat};
			\addplot[  dashed,gray!100,ultra thick]   file {Data/beta_nofix_s2.dat};\addplot[dash dot,black!50, ultra thick]  file {Data/ell1_nofix_s2.dat};
	\addplot[dash dot,black, ultra thick]  file {Data/ell2_nofix_s2.dat};
		\addplot[dotted,green, thick]  file {Data/x_2.dat};
	\addplot[dotted,blue,  thick]  file {Data/x_5.dat};
		\addplot[dotted,magenta, thick]  file {Data/x_8.dat};
		\addlegendentry{$\alpha(t)$};
		\addlegendentry{$\beta(t)$};
		\addlegendentry{$\ell_1(t)$};
		\addlegendentry{$\ell_2(t)$};
		\addlegendentry{$x_1^*(t)$};
		\addlegendentry{$x_2^*(t)$};
		\addlegendentry{$x_3^*(t)$};
		\addplot[mark=*,color=green,mark size=1pt] coordinates {(0,2)};
		\addplot[mark=*,color=blue, mark size=1pt] coordinates {(0,5)};
		\addplot[mark=*,color=magenta, mark size=1pt] coordinates {(0,8)};
	\end{axis}
 \end{tikzpicture}
    \caption{Evolution of the intervention region for the parameters in Table \ref{FBNEtab:Parameters}.}
    \label{fig:s2}
\end{figure}

 In Figure \ref{fig:valuefunctions2}, we can see that the value functions of Player 1 and Player 2 at the initial time are quadratic in state when the state is in $(\ell_1(0),\ell_2(0))$, and that, outside this region, the value functions are linear in  state. The value function of Player 1 jumps at $\ell_1(0)$ and $\ell_2(0)$ whereas Player 2's value function is continuous for all initial state values.
 
\begin{figure}[htbp]
    \centering
 \begin{tikzpicture}
    \begin{axis}[ xlabel=$x_0$ ,xmin=0,xmax=10.0305,ymin=-1,ymax=30,legend style={overlay},  legend columns=2,
		legend pos=north west,scale=1]
		\addplot[black,  thick] file {Data/V11_s2.dat};
			\addplot[gray!60,  thick]   file {Data/V21_s2.dat};
			\addplot[dash dot,ultra thick, samples=50, smooth,black!50] coordinates {( 3.3822,-1)( 3.3822,20)};
\addplot[dashed,ultra thick, samples=50, smooth,gray!30] coordinates {( 4.5111,-1)( 4.5111,20)};
\addplot[dashed,ultra thick, samples=50, smooth,gray!100] coordinates {(  5.5731,-1)(5.5731,20)};
\addplot[dash dot,ultra thick, samples=50, smooth,black] coordinates {(   7.0305,-1)( 7.0305,20)};
\addlegendentry{$V_1(0)$};
		\addlegendentry{$V_2(0)$}
			\addlegendentry{$\ell_1(0)$};	\addlegendentry{$\alpha(0)$};
		\addlegendentry{$\beta(0)$};
		\addlegendentry{$\ell_2(0)$};
\addplot[black,  thick] file {Data/V12_s2.dat};
			\addplot[gray!60,  thick]   file {Data/V22_s2.dat};
			\addplot[black,  thick] file {Data/V13_s2.dat};
			\addplot[gray!60,  thick]   file {Data/V23_s2.dat};
	\end{axis}
 \end{tikzpicture}
    \caption{Value function for the parameters in Table \ref{FBNEtab:Parameters}.}
    \label{fig:valuefunctions2}
\end{figure}
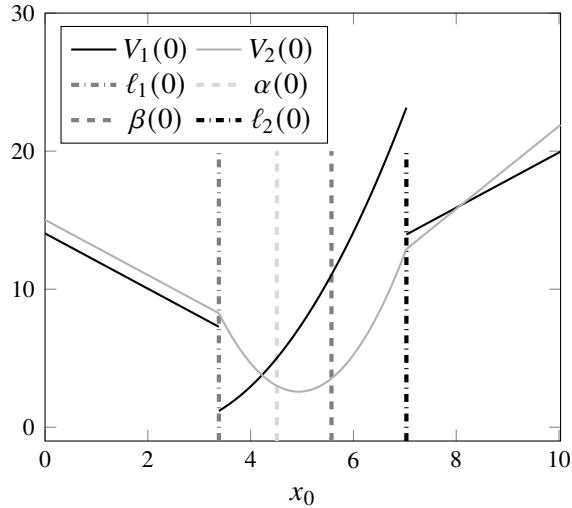

Next, we consider the case where the penalty associated with the state deviating from the target value at the terminal time is the same as the running cost. Therefore, in Figure \ref{fig:s2w2equal}, we can see that $\ell_1(\cdot)$, $\alpha(\cdot)$, $\beta(\cdot)$, and $\ell_2(\cdot)$ are a further away from the target state of Player 2 near the initial time, as compared to Figure \ref{fig:s2}. Here, $x_1^*(t), \, x_2^*(t),$ and $x_3^*(t)$ denote the equilibrium evolution of the state trajectory for initial state values of $1$, $6$, and $10$, respectively. The  value functions of Player 1 and Player 2 at the initial time are given in Figure \ref{fig:valuefunction} for different values of the initial state.
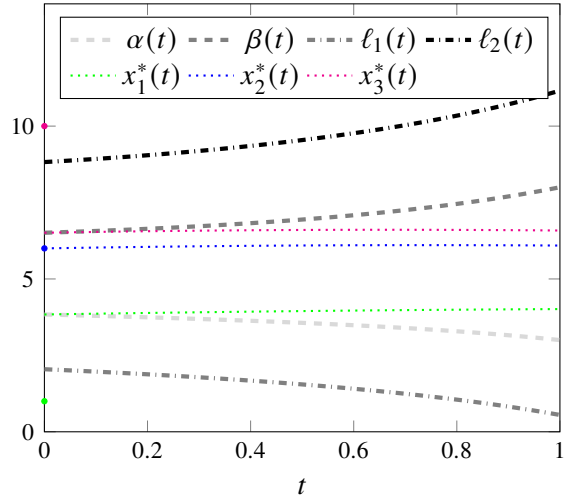
\begin{figure}[htbp]
    \centering
 \begin{tikzpicture}
    \begin{axis}[ xlabel=$t$,xmin=0,xmax=1,ymin=0,ymax=14,legend style={overlay},  legend columns=4,
		legend pos=north west,scale=1]
		\addplot[dashed,gray!30, ultra thick] file {Data/alpha_nofix.dat};
			\addplot[ dashed,gray!100, ultra thick]   file {Data/beta_nofix.dat};\addplot[dash dot,black!50, ultra thick]  file {Data/ell1_nofix.dat};
	\addplot[dash dot,black, ultra thick]  file {Data/ell2_nofix.dat};
	\addplot[dotted,green, thick]  file {Data/xw1_1.dat};
	\addplot[dotted,blue,  thick]  file {Data/xw1_6.dat};
		\addplot[dotted,magenta, thick]  file {Data/xw1_10.dat};
			\addlegendentry{$\alpha(t)$};
		\addlegendentry{$\beta(t)$};
		\addlegendentry{$\ell_1(t)$};
		\addlegendentry{$\ell_2(t)$};	\addlegendentry{$x_1^*(t)$};			\addlegendentry{$x_2^*(t)$};
								\addlegendentry{$x_3^*(t)$};
		\addplot[mark=*,color=green,mark size=1pt] coordinates {(0,1)};
		\addplot[mark=*,color=blue, mark size=1pt] coordinates {(0,6)};
		\addplot[mark=*,color=magenta, mark size=1pt] coordinates {(0,10)};
	\end{axis}
 \end{tikzpicture}
    \caption{Evolution of the intervention region for the parameters in Table \ref{FBNEtab:Parameters} with $w_2=1$.}
    \label{fig:s2w2equal}
\end{figure}

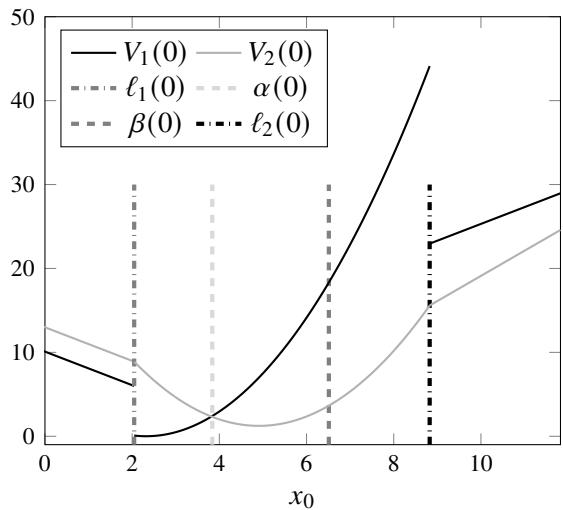
\begin{figure}[htbp]
    \centering
 \begin{tikzpicture}
    \begin{axis}[ xlabel=$x_0$ ,xmin=0,xmax=11.8240,ymin=-1,ymax=50,legend style={overlay},  legend columns=2,
		legend pos=north west,scale=1]
		\addplot[black, thick] file {Data/V11.dat};
			\addplot[gray!60, thick]   file {Data/V21.dat};
			\addplot[dash dot,ultra thick, samples=50, smooth,black!50] coordinates {( 2.0468,-1)( 2.0468,30)};
\addplot[dashed,ultra thick, samples=50, smooth,gray!30] coordinates {( 3.8380,-1)( 3.8380,30)};
\addplot[dashed,ultra thick, samples=50, smooth,gray!100] coordinates {(6.5116,-1)(6.5116,30)};
\addplot[dash dot,ultra thick, samples=50, smooth,black] coordinates {(8.8240,-1)(8.8240,30)};
\addlegendentry{$V_1(0)$};
		\addlegendentry{$V_2(0)$}
			\addlegendentry{$\ell_1(0)$};	\addlegendentry{$\alpha(0)$};
		\addlegendentry{$\beta(0)$};
		\addlegendentry{$\ell_2(0)$};
		\addplot[black, thick] file {Data/V12.dat};
			\addplot[gray!60, thick]   file {Data/V22.dat};
			\addplot[black, thick] file {Data/V13.dat};
			\addplot[gray!60, thick]   file {Data/V23.dat};
	\end{axis}
 \end{tikzpicture}
    \caption{Value function for the parameters in Table \ref{FBNEtab:Parameters} with $w_2=1$.}
    \label{fig:valuefunction}
\end{figure}
\section{Conclusions}\label{FBNEsec:conclusion}
In this paper, we  considered a two-player finite-horizon nonzero-sum differential game where Player 1 uses piecewise-continuous controls and Player 2 uses impulse controls. We determined an upper bound on the equilibrium number of impulses  and provided sufficient conditions to characterize the feedback Nash equilibria for this general class of differential games with impulse controls. The sufficient conditions are given as a coupled system of Hamilton-Jacobi-Bellman equations with jumps and quasi-variational inequalities. To the best of our knowledge, this is the first characterization of feedback Nash equilibrium in differential games with impulse controls where  at least one player uses piecewise-continuous controls. In this, our paper also differs from earlier papers on impulse games where equilibrium solutions were derived for problems in which both players use impulse controls only. Furthermore, we extended a well-studied linear-quadratic impulse control problem to a game setting where both players use their controls to minimize  the cost associated with the state deviating from their target values.  

We obtained closed-form solutions for the feedback Nash equilibrium in the scalar linear-quadratic differential game
based on certain regularity assumptions on the value function that have been
assumed in the literature  (see e.g., \cite{Bertola2016} and \cite{Aid2020}). In 
future work, we plan to relax these assumptions and develop policy iteration-type
algorithms  \cite{Bokanowski2009} that can
solve the quasi-variational inequalities for the impulse player in the general class of differential games with impulse control.

	\appendix
\section{Appendix}
\subsection{Proof of Proposition \ref{FBNEprop_bound_number}}\label{FBNEproof:prop:number:upper}

A feasible strategy of Player $2$ is not to give any impulse in $[0,T]$ so that \begin{align}
\sum_{j\geq 1}\mathbbm{1}_{t\leq\tau_j\leq T}~b_2(x(\tau_j),\xi_j)=0,
\end{align} and it follows from the boundedness of $h_2$ and $s_2$ in Assumption \ref{FBNEassum:lipschitz} that
\begin{align*}
V_2(t,x)\leq & \int_{t}^{T}h_2(x(s),\gamma^*(s,x(s)))ds +s_{2}(x(T))\leq \|h_2\|_{\infty}(T-t)+\|s_2\|_{\infty}.
\end{align*}
Next, for any $\epsilon>0$, we choose a strategy ${\delta}_{[t,T]}\in \Delta_{[t,T]}$ so that 
\begin{align*}
    V_2(t,x)+\epsilon & >J_2(x,\gamma^*_{[t,T]},{\delta}_{[t,T]})\geq -\|h_2\|_{\infty}(T-t)-\|s_2\|_{\infty},
\end{align*}
where the second inequality follows from Assumption \ref{FBNEassum:lipschitz}. 
This proves that the value function is bounded such that
\begin{align}
\lvert V_2(t,x) \rvert \leq \|h_2\|_{\infty}(T-t)+\|s_2\|_{\infty},\; \forall (t,x)\in [0,T]\times\mathbb{R}^n. \label{FBNEeq:lem:bounded:value}
\end{align}

For any $\epsilon>0$, consider an $\epsilon$-optimal strategy ${v}^\epsilon$ with $N({v}^\epsilon)$ impulses. From the boundedness of  $h_2$, we obtain
\begin{align*}
    V_2(t,x)+\epsilon &> J_2(x,\gamma_1,{v}^\epsilon) \geq -\|h_2\|_{\infty}(T-t)+\mu N({v}^\epsilon)-\|s_2\|_{\infty}.
\end{align*}
Using the above relation and  \eqref{FBNEeq:lem:bounded:value}, we obtain
\begin{align*}
   &-\|h_2\|_{\infty}(T-t)+\mu N({v}^\epsilon)-\|s_2\|_{\infty}<\|h_2\|_{\infty}(T-t)+\|s_2\|_{\infty} +\epsilon.
\end{align*}
Since $\mu>0$, we can rewrite the above inequality as follows:
\begin{align*}
   N({v}^\epsilon)< \frac{2\left(\|h_2\|_{\infty}(T-t)+\|s_2\|_{\infty}\right)+\epsilon}{\mu}.
\end{align*}
Since $\epsilon>0$ is arbitrarily chosen for an $\epsilon$-optimal strategy of Player 2, the upper bound $K$ on the number of impulses is given by \eqref{eq:bound:number:impulse} as $\epsilon \to 0$.

\noindent
For a feasible strategy of Player 1 given by $\gamma(t,x)=0$ for all $(t,x)\in \Sigma$ and the upper bound $K$ on the number of impulses, we have 
\begin{align*}
V_1(t,x)&\leq \int_{t}^{T}h_1(x(s),0)ds +\sum_{j\geq1} \mathbbm{1}_{t\leq \tau_j<T}~b_1(x(\tau_i),\xi_j)+s_{1}(x(T))\\ &\quad\leq  \int_{t}^{T}h_1(x(s),0)ds+K\|b_1\|_{\infty}+s_{1}(x(T))\\&\qquad\leq \|h_1\|_{\infty}(T-t)+K\|b_1\|_{\infty}+\|s_1\|_{\infty},
\end{align*}
where the last inequality follows from the boundedness of $b_1$ and $s_1$ in Assumption \ref{FBNEassum:lipschitz}.
For any $\epsilon>0$, we take a strategy ${\gamma}_{[t,T]}\in \Gamma_{[t,T]}$ so that 
\begin{align*}
 &   V_1(t,x)+\epsilon>J_1(x,{\gamma}_{[t,T]},\delta^*_{[t,T]})\geq -\|h_1\|_{\infty}(T-t)-K\|b_1\|_{\infty}-\|s_1\|_{\infty}.
\end{align*}
This proves that the value function of Player 1 is bounded.  

\subsection{Proof of Theorem \ref{theorem:verif:lqdg}}\label{proof:theorem:verif:lqdg}
 From \eqref{FBNEalpha2diffquad}, we have $\frac{\partial V_2(t,\alpha(t))}{\partial x}=-c$ and $\frac{\partial V_2(t,\beta(t))}{\partial x}=d$. Using the strict convexity of $V_2$ in $x$ for $(t,x)\in \mathcal{C}$ (Assumption \ref{FBNEV2:convex}), we obtain $$-c<\frac{\partial V_2(t,x)}{\partial x}<d, \;\forall (t,x): x\in (\alpha(t),\beta(t)).$$
 Therefore, $\mathcal{R}V_2(t,x)=\Phi_2(t,x)+\min(C,D)$ when the time and state pairs $(t,x)$ are such that $x\in (\alpha(t),\beta(t))$. \\

\noindent
When $x\in (\ell_1(t),\alpha(t))$, we have $\frac{\partial V_2(t,x)}{\partial x}\leq-c$ and,  for $x\in (\beta(t),\ell_2(t))$, we obtain $\frac{\partial V_2(t,x)}{\partial x}\geq d$ from the strict convexity of $V_2(t,x)$ in $x\in (\ell_1(t),\ell_2(t))$. Therefore, the operator $\mathcal{R}$ satisfies the following system:
\begin{align}
    \R V_2(t,x)=\begin{cases}
    \Phi_2(t,\alpha(t))+C+c(\alpha(t)-x) &x \leq \alpha(t)\\
  \Phi_2(t,x)+\min(C,D) & x\in (\alpha(t),\beta(t))\\
     \Phi_2(t,\beta(t))+D+d(x-\beta(t)) & x \geq \beta(t).
     \end{cases}
\end{align}
Clearly, $V_2-\mathcal{R}V_2<0$ in the continuation region and $V_2(t,x)=\mathcal{R}V_2(t,x)$ in the intervention region.

\noindent
Next, we derive the conditions under which the value function of Player 2 satisfies \eqref{FBNEQVI:HJB}.
For $x<\ell_1(t)$, we have 
\begin{align}
    V_2(t,x)= \Phi_2(t,\alpha(t))+C+c(\alpha(t)-x).
\end{align}
When $x < \ell_1(t)$, we obtain
\begin{align*}
    &\frac{\partial V_2(t,x)}{\partial t}+\mathcal{H}_2(x,\gamma^*(t,x), \frac{\partial V_2(t,x)}{\partial x})\\&\quad=\frac{\partial V_2(t,x)}{\partial t}+\frac{1}{2}w_2(x-\rho_2)^2+\frac{\partial V_2(t,x)}{\partial x}(ax+\mathbbm{1}_{\ell_1(t)<x<\ell_2(t)}~b\gamma^*(t,x))\\&\qquad=\frac{\partial \Phi_2(t,\alpha(t))}{\partial t}+\left(\frac{\partial \Phi_2(t,\alpha(t))}{\partial x}+c\right)\frac{d\alpha(t)}{dt}-cax+\frac{1}{2}w_2x^2+\frac{1}{2}w_2\rho_2^2-w_2x\rho_2.
\end{align*}
\noindent
Substituting \eqref{FBNEalpha2diffquad} in the above equation, we get the roots of the above equation as follows:
\begin{align}
    &x_{11}(t),x_{12}(t)=\frac{(ca+w_2\rho_2)\pm\sqrt{\theta_\alpha(t)} }{w_2},
\end{align}
where $x_{11}(t)<x_{12}(t)$, and $\theta_{\alpha}(t)$ is given by equation \eqref{FBNEeq:delalpha}.  
Therefore, \eqref{FBNEQVI:HJB} holds if  $\ell_1(t)\leq x_{11}(t)$ and $\theta_{\alpha}(t)\geq 0$ for all $t\in [0,T]$.

\noindent
For $x>\ell_2(t)$, we obtain
\begin{align*}
    &\frac{\partial V_2(t,x)}{\partial t}+\mathcal{H}_2(x,\gamma^*(t,x), \frac{\partial V_2(t,x)}{\partial x})\\&\quad=\frac{\partial V_2(t,x)}{\partial t}+\frac{1}{2}w_2(x-\rho_2)^2+\frac{\partial V_2(t,x)}{\partial t}(ax+\mathbbm{1}_{\ell_1(t)<x<\ell_2(t)}b\gamma^*(t,x))\\&\qquad=\frac{\partial \Phi_2(t,\beta(t))}{\partial t}+\left(\frac{\partial \Phi_2(t,\beta(t))}{\partial x}-d\right)\frac{d\beta(t)}{dt}+dax +\frac{1}{2}w_2x^2+\frac{1}{2}w_2\rho_2^2-w_2x\rho_2.
\end{align*}
On substituting \eqref{FBNEalpha2diffquad} in the above equation, we obtain the roots of the above equation as follows:
\begin{align}
    &x_{21}(t),x_{22}(t)=\frac{-(da-w_2\rho_2)\pm\sqrt{\theta_\beta(t)} }{w_2},
\end{align}
where $x_{21}(t)<x_{22}(t)$ and $\theta_{\beta}(t)$ is given by \eqref{FBNEeq:delbeta}. 
Therefore, \eqref{FBNEQVI:HJB} holds if  $\ell_2(t)\geq x_{22}(t)$.

\subsection{Analytical solution of ODE}\label{FBNEsolve31}
To solve the differential equation $
\dot{p_1}(t)+b_x(p_1(t))^2+2ap_1(t)+w_1=0$ for $t\in (\tau_i,\tau_{i+1}), i\in \{0,1,\cdots, k\}$, we substitute $p_1(t)=\frac{\dot{\mu}(t)}{b_x\mu(t)}$ to obtain a second-order ordinary differential equation $\ddot{\mu}(t)+2a\dot{\mu}(t)+b_xw_1\mu(t)=0$. When $\theta=2\sqrt{a^2-w_1b_x}$, the solution of this equation is 
\begin{align*}
\mu(t)=e^{-at}(F_1e^{\frac{1}{2}\theta t}+F_2e^{-\frac{1}{2}\theta t}),    
\end{align*}
where $F_1$ and $F_2$ are constants.
\noindent
So, $p_1(t)$ is given by
\begin{align*}
   p_1(t)=\frac{\dot{\mu}(t)}{b_x\mu(t)}=  &\frac{-a\mu(t)+\frac{\theta}{2}e^{-at}(F_1e^{\frac{1}{2}\theta t}-F_2e^{-\frac{1}{2}\theta t})}{b_xe^{-at}(F_1e^{\frac{1}{2}\theta t}+F_2e^{\frac{1}{2}\theta t}) }\\
   &=\frac{1}{b_x}\left(-a+\frac{\theta}{2}-\frac{\theta}{C_1e^{\theta t}+1 }\right).
\end{align*}
Substitute $p_1(T)=s_1$ in the above equation to obtain
\begin{align}
    C_{1}=\left(\frac{2\theta}{\theta-2b_xs_1-2a}-1\right)e^{-\theta T}.
\end{align} 	
\section*{Acknowledgement}
The first author's research is supported by the FRQNT Doctoral research scholarship (B2X, 275596). The second author's research is supported by SERB, Government of India, grant MTR/2019/000771.

\end{document}